\documentclass[a4paper,fleqn]{cas-sc}
\usepackage[numbers]{natbib}

\newtheorem{theorem}{Theorem}
\newtheorem{lemma}[theorem]{Lemma}
\newdefinition{rmk}{Remark}
\newproof{pf}{Proof}

\newtheorem{proposition}[theorem]{Proposition}
\newtheorem{corollary}[theorem]{Corollary}
\newtheorem{question}{Question}

\newcommand{\LT}{\widetilde L}
\newcommand{\cL}{\mathcal L}
\newcommand{\R}{\mathbb R}
\newcommand{\C}{\mathbb C}

\DeclareMathOperator{\conv}{conv}
\DeclareMathOperator{\supp}{supp}
\def\x#1{} 
\def\cdc{,\ldots,}
\def\Up#1{\vspace*{-#1em}}
\def\tu#1{\textup{#1}}
\definecolor{nb}{rgb}{0,0,.7} 

\newcommand{\nb}[1]{#1}
\newcommand{\nH}[1]{#1} 
\newenvironment{nbblock}{}{}

\newcommand{\nvv}[1]{{\color{nb}#1}} 
\date{}

\begin{document}
\let\WriteBookmarks\relax
\def\floatpagepagefraction{1}
\def\textpagefraction{.001}
\shorttitle{Digraph Laplacian spectra: The 2005 polygon conjecture fails for all orders above four}
\shortauthors{P. Chebotarev}

\title [mode = title]{Digraph Laplacian spectra: The 2005 polygon conjecture fails for all orders above four}
\tnotemark[1] 

\tnotetext[1]{This work was supported by the Israel Science Foundation (grants 1225/20 and 907/24)}

\author[1]{Pavel Chebotarev}[orcid=0000-0001-8232-847X]    

\ead{pavel4e@gmail.com; pavel4e@technion.ac.il}


\affiliation[1]{organization={Department of Mathematics, Technion---Israel Institute of Technology},
                city={Haifa},
                postcode={3200003},
                country={Israel}}

\begin{abstract}
Let \nb{$\Lambda_n$ be the set of all eigenvalues of standardized Laplacian matrices of order $n$, equivalently, of stochastic $n\times n$ matrices with off-diagonal entries at most $\frac1n$. In 2005, a polygon $S_n\subseteq\Lambda_n$ was constructed, and it was conjectured that $\Lambda_n=S_n$; this holds for $n\le3$. We show that the conjecture fails for every $n\ge5$. The vertices of $S_n$ are eigenvalues of circulant Laplacians built from powers of the directed $n$-cycle. To first order, no perturbation of these circulants moves such an eigenvalue out of $S_n$; however, adding chords with nonconstant weights moves it along an edge of $S_n$ to first order and, with an explicit coefficient, beyond the supporting line of this edge to second order. Consequently, every edge of $S_n$ not incident to $0$ or $1$ contains interior points of $\Lambda_n$, except possibly, for even $n$, the two horizontal edges, whereas the four edges incident to $0$ or $1$ lie on $\partial\Lambda_n$.
\nvv{For $n=4$, whose only other edges are horizontal, the question has recently been settled, $\Lambda_4=S_4$, by computer-assisted proofs of D.\,Z.~Liu and L.~Marshall; Liu also proved that $\Lambda_n$ is not convex for $n\ge5$.}
We certify an explicit weighted counterexample for $n=5$, lying only $1.5\cdot10^{-3}$ outside $S_5$, and an unweighted one for $n=7$: the directed $6$-cycle with a duplicated vertex. Its eigenvalue outside $S_7$ is also an eigenvalue of the two-chord digraph that ends the above perturbation; for the analogous digraphs of order $n$, the excess is asymptotic to $\frac{\pi^2}{2n^4}$. We also observe that $S_n$ is a zonotope, describe $\Lambda_n$ in terms of eigenvectors, and prove that $S_m\subseteq S_n$ whenever $m\mid n$.}
\end{abstract}


\maketitle

\section{Introduction}

Spectral localization for nonsymmetric Laplacian matrices is closely related to the classical Karpelevi\v{c} problem of determining the possible eigenvalues of stochastic matrices \cite{JohnsonPaparella017,Karpelevich51E,KirklandLaffeyS020}.

\nb{To see this connection, recall that a \emph{standardized Laplacian matrix} $\LT$ of order $n$ is a real matrix with zero row sums and
off-diagonal entries in $\bigl[-\frac1n,0\bigr]$ (see Section~\ref{s:polygzono}). Let $\Lambda_n$ denote the set of all eigenvalues of such matrices. Then $P=I-\LT$ is stochastic, and every stochastic $n\times n$ matrix whose off-diagonal entries do not exceed $1/n$ arises in this way. Since $\sigma(P)=1-\sigma(\LT)$, the eigenvalue region of this stochastic subclass is $1-\Lambda_n$, the reflection of $\Lambda_n$ about $1/2$.

However, by the complementary-digraph symmetry \cite[Theorem~1]{AgaChe05LAA}, $\Lambda_n$ is invariant under reflection about $1/2$: $1-\Lambda_n=\Lambda_n$. Thus the two eigenvalue regions coincide, and $\Lambda_n\subseteq\Theta_n$, where $\Theta_n$ is the Karpelevi\v{c} region of all possible eigenvalues of stochastic $n\times n$ matrices.}
A polygon $S_n\subseteq\Lambda_n$ was constructed in~\cite{AgaChe05LAA}, where it was also conjectured that this inclusion is in fact an equality:
\[
S_n=\Lambda_n.
\]

To the best of our knowledge, this conjecture \nvv{had} remained unresolved since its formulation in 2005.

The polygon $S_n$ has a natural cyclic origin. Circulant matrices, including circulant Laplacians, are polynomials in the adjacency matrix of the directed $n$-cycle, so they share the eigenvector $(1,\mu,\mu^2,\dots,\mu^{n-1})^{\mathsf T}$,
$\mu=e^{-2\pi i/n}$, whose components are the $n$th roots of unity. The vertices
of $S_n$ are the eigenvalues, for this common eigenvector, of certain circulant
standardized Laplacians. We first make this geometry explicit by showing that $S_n$ is a zonotope. This parametrizes the circulants realizing its points and leads to a description of $\Lambda_n$ in terms of eigenvectors: 
\begin{equation}\label{e:union}
\Lambda_n=\bigcup_{v\in\C^n\setminus\{0\}}\;\bigcap_{i\in\supp v} Z_i(v),\quad\supp v=\{i: v_i\ne0\},
\end{equation}
where for each nonzero $v\in\C^n$, the $Z_i(v)$ are certain zonotopes, one for each vertex~$i$, built from the ratios $v_j/v_i$ (see Section~\ref{s:eigvec}); their intersection consists of exactly those $\lambda$ for which $v$ is an eigenvector, with eigenvalue $\lambda$, of some
standardized Laplacian matrix.
For the common eigenvector of circulants, and more generally for any vector $v_0$ whose components are the $n$ distinct $n$th roots of unity, every $Z_i(v_0)$ is exactly~$S_n$. Thus the polygon conjecture asserts that this vector is extremal:
$\bigcap_{i\in\supp v}Z_i(v)\subseteq S_n$ for every nonzero $v\in\C^n$.

That extremality assertion is false.
For each $n\ge4$, we define a digraph $\Gamma_n$ by duplicating one vertex of the directed $(n-1)$-cycle; this keeps the cyclic structure but destroys circulant symmetry.
The standardized Laplacian of $\Gamma_7$ has an eigenvalue lying strictly outside $S_7$. Since the excess is only about $2.1\cdot10^{-4}$, \ref{a:certification} gives a rigorous certification separating this eigenvalue from $S_7$, rather than relying on floating-point computation. For the same construction, if $\delta_n$ denotes the signed excess of the corresponding eigenvalue beyond the relevant supporting line of $S_n$, then
\[
\delta_n
=
\frac{\pi^2}{2n^4}
-
\frac{9\pi^2}{4n^5}
+
O\bigl(n^{-6}\bigr).
\]
Hence $S_n\subsetneq\Lambda_n$ for all sufficiently large $n$. Numerically, the same construction gives $\delta_n>0$ for every integer $7\le n\le300$.

\nb{A second mechanism gives counterexamples for every $n\ge5$ (Section~\ref{s:chords}). The standardized Laplacian of the directed $n$-cycle has the vertex $\lambda_1$ of $S_n$ as an eigenvalue. Adding chords $k\to k+2$ with weights $t\omega_k$ moves this eigenvalue, to first order in~$t$, along the edge $[\lambda_1,\lambda_2]$; its signed distance from the supporting line has the second-order term 
$\frac2n\sin^2\frac{2\pi}n\,\gamma(\omega)\,t^2$, where $\gamma(\omega)$ is the cyclic autocovariance of the weights at lag~$1$. For two consecutive chords of equal weight, $\gamma(\omega)>0$ exactly when $n\ge5$. For $n=5$, a certified explicit example has an eigenvalue at distance $1.537\ldots\cdot10^{-3}$ from~$S_5$, about seven times the distance for~$\Gamma_7$. At $t=1$, the two-chord digraph has the same eigenvalues as $\Gamma_n$ except that $\frac1n$ is replaced by~$\frac2n$; numerically, for $n=7$, the branch issuing from $\lambda_1$ ends at the exterior eigenvalue of $\Gamma_7$. 
The same mechanism, applied to the union of the first $m$ powers of the Hamiltonian cycle with chords $k\to k+m+1$, shows that every edge of $S_n$ not incident to $0$ or~$1$ contains interior points of $\Lambda_n$, except possibly, for even~$n$, the two horizontal edges, at which the second-order term never points outward (Section~\ref{ss:edges}). By contrast, the four edges incident to $0$ or~$1$ lie on $\partial\Lambda_n$ \cite[Theorems~5 and~6]{AgaChe05LAA}.}

We also show that $S_m\subseteq S_n$ whenever $m\mid n$, with $S_m$ inscribed in $S_{2m}$ and $S_m\setminus\{0,1\}$ interior to $S_n$ if $m\mid n$ and $n\ge3m$; consequently, $\bigl(\bigcup_n S_n\bigr)\setminus\{0,1\}$ is open and convex. Moreover, $\Lambda_n$ is a compact semialgebraic set, star-shaped with respect to every point of $[0,1]$.

\nb{Since $\Lambda_n=S_n$ for $n\le3$ \cite{AgaChe05LAA}, our results leave the conjecture open only for $n=4$. For this order, our computations found no counterexample, and the second-order term above never points outward. The case $n=4$ has recently been settled in the affirmative (see below).}

\nvv{Several further results have been obtained by D.\,Z.~Liu~\cite{Liu026v3}.}
\nH{Most notably, he proved that $\Lambda_n$ is not convex for every $n\ge5$. Near the vertex $\lambda_1$, every point $z\in\Lambda_n$ satisfies $d(z)\le C|z-\lambda_1|^2$. The eigenvalues $\lambda(t)$ of Theorem~\ref{t:edges} (with $m=1$) satisfy this bound, since
$d(\lambda(t))$ is of order~$t^2$ and $|\lambda(t)-\lambda_1|$ of order~$t$. However, along the segment $[\lambda_1,\lambda(t)]$ the function $d$ grows linearly with the distance from~$\lambda_1$, so the points of this segment close to $\lambda_1$ violate the bound. Hence, for small $t>0$, the segment $[\lambda_1,\lambda(t)]$, whose endpoints lie in~$\Lambda_n$, is not contained in~$\Lambda_n$. In particular, $\Lambda_n$ is not star-shaped with respect to~$\lambda_1$. The proof combines a quantitative form of Proposition~\ref{p:first} at~$\lambda_1$, the fact that only the relabelings of the standardized Laplacian of the directed $n$-cycle have the eigenvalue $\lambda_1$ (a consequence of \cite[Remark~2]{AgaChe05LAA}, based on~\cite{DmitrievDynkin46E}), and a compactness argument. This answers the convexity part of Question~\ref{q:shape}. Liu also showed that at the horizontal edges the second-order term never points outward along any $C^2$ curve of standardized Laplacians through the corresponding circulant, which extends Remark~\ref{r:central}, and found explicit digraphs with three fractional chord weights for $n=5$ and $n=6$ whose eigenvalues lie at distances $2.47\cdot10^{-3}$ and $1.79\cdot10^{-3}$ beyond $[\lambda_1,\lambda_2]$. Finally, Liu \nvv{gave (Theorem~1)} a computer-assisted proof \nvv{(first posted 25 October 2026)} that $\Lambda_4=S_4$, based on a linear-programming reduction to digraphs with at most six fractional weights and on exact Bernstein expansions. Independently, L.~Marshall~\cite{Marshall026} (Zenodo, posted 28 October 2026) \x{published}\nvv{presented} a computer-assisted proof of $\Lambda_4=S_4$ with open certificates. It reduces the question, via \cite[Theorems~5 and~6]{AgaChe05LAA} and a linear-programming argument, to the nonnegativity of an explicit polynomial on $836$ classes of six-dimensional faces of the weight cube. We have checked the reduction, and the certificates have been replayed by two independent checkers, one of them using only exact rational arithmetic in the Python standard library. Thus $\Lambda_n=S_n$ holds exactly for
$n\le4$, and, by Liu's \nvv{result}, $\Lambda_n$ is convex exactly for $n\le4$. Neither work has yet been peer reviewed.}

\nvv{In a companion preprint~\cite{Che026b}, we study how $\Lambda_n$ leaves $S_n$ near an edge $[\lambda_m,\lambda_{m+1}]$. There we prove a switching law for the matrices that realize boundary points of~$\Lambda_n$. We also show that, for the circulant realizing the point $\lambda_m+\tau g_{m+1}$ of the edge, the first Fourier harmonic of a perturbation of the chord weights in~\eqref{e:Lmw} moves the eigenvalue beyond the line of the edge at second order if and only if $\tau$ is below an explicit threshold~$\tau^*_m$. Finally, we give a conjectural, numerically supported description of $\partial\Lambda_5$.}

\nvv{The problem of describing the precise geometry of $\Lambda_n$ beyond $S_n$ remains open.}

\section{The polygon $S_n$ and its zonotope representation}\label{s:polygzono}

A \emph{standardized Laplacian matrix} of order $n$ is a real matrix $\LT=(\ell_{ij})$ with zero row sums and
$-\tfrac1n\le\ell_{ij}\le 0$ for $j\ne i$.
For all $i,j=1\cdc n,\,j\ne i,$ put $a_{ij}=-\ell_{ij}\in\bigl[0,\tfrac1n\bigr].$ Then we have
\[
(\LT v)_i=\sum_{j\ne i}a_{ij}\,(v_i-v_j),\quad v\in\C^n,\;\, i=1,\ldots,n.
\]
Let $\cL_n$ be the set of such matrices and $\Lambda_n=\bigcup_{\LT\in\cL_n}\sigma(\LT)$, where $\sigma(\LT)$ is the spectrum of~$\LT$.
We study the properties of $\Lambda_n$ for $n\ge2$.

Standardized Laplacian matrices are normalized Laplacian matrices of weighted digraphs\nvv{; for the algebraic properties of the latter, including their relation to matrices of spanning forests, see~\cite{CheAga02a+}}. Fix $b>0$, and let $\Gamma$ be a weighted digraph on $n$ vertices with arc weights in $(0,b]$. Its Laplacian matrix $L(\Gamma)=(\ell_{ij})$ is defined by setting, for $j\ne i$, $-\ell_{ij}$ equal to the weight of the arc $(i,j)$ if $\Gamma$ has this arc and $0$ otherwise, the diagonal entries being determined by zero row sums. Then $\LT(\Gamma)=(nb)^{-1}L(\Gamma)$ belongs to $\cL_n$, and, conversely, every $\LT\in\cL_n$ arises in this way: $\Gamma$ has an arc $(i,j)$ of weight $nb\,a_{ij}$ whenever $a_{ij}>0$. Thus $(\LT v)_i$ is a weighted sum over the out-neighbors $j$ of vertex~$i$, and the matrices in $\cL_n$ with off-diagonal entries in $\bigl\{0,-\frac1n\bigr\}$ are exactly the standardized Laplacians of digraphs with all arc weights equal to~$b$. For such digraphs, $\LT(\Gamma)$ does not depend on~$b$ and is the standardized Laplacian of the underlying unweighted digraph.

Put $\mu=e^{-2\pi i/n}$ and, as in~\cite[Section~4]{AgaChe05LAA},
\begin{equation}\label{lambda_k}
\lambda_k=\frac1n\bigl(k-\mu-\mu^2-\dots-\mu^k\bigr),\qquad k=0,1,\dots,n-1.
\end{equation}

Let $S_n=\conv\bigl\{\lambda_k,\bar\lambda_k: 0\le k\le n-1\bigr\}$.
It is proved in~\cite[Theorem~6]{AgaChe05LAA} that $S_n\subseteq\Lambda_n$, and
it is conjectured~\cite[Conjecture~1]{AgaChe05LAA} that $\Lambda_n=S_n$.

Let $K\subset\C$ be a nonempty compact convex set. A line is a
\emph{supporting line} of $K$ if it meets $K$ and $K$ lies in one of the two
closed half-planes bounded by it; this half-plane is a \emph{supporting half-plane} of~$K$, and the unit normal to the line pointing away from it is the \emph{outward normal}. For each direction $e^{i\theta}$, there is exactly one supporting line with outward normal $e^{i\theta}$, namely $\bigl\{z:\operatorname{Re}\bigl(e^{-i\theta}z\bigr)=h_K(\theta)\bigr\}$, where
\[
h_K(\theta)=\max_{z\in K}\operatorname{Re}\bigl(e^{-i\theta}z\bigr)
\]
is the \emph{support function} of~$K$ (see, e.g., \cite{Schneider14}); thus $h_K(\theta)$ is the signed distance from the origin to this line. Moreover, $K$ is the intersection of its supporting half-planes:
\[
K=\bigcap_{\theta}\bigl\{z:\operatorname{Re}\bigl(e^{-i\theta}z\bigr)\le h_K(\theta)\bigr\},
\]
so that, for compact convex $K,K'$, one has $K\subseteq K'$ if and only if $h_K\le h_{K'}$.

A \emph{zonotope} in $\R^d$ is a translate of a Minkowski sum of finitely many segments $[0,g]$, $g\in\R^d$ (see, e.g., \cite{Schneider14}). All zonotopes below lie in $\C\cong\R^2$ and have 
the form $\sum_m[0,g_m]=\bigl\{\sum_m t_mg_m: 0\le t_m\le1\bigr\}$ with $g_m\in\C$; the vectors $g_m$ are called the \emph{generators}. Such a zonotope is centrally symmetric about $\frac12\sum_m g_m$.

Since support functions add under Minkowski sums and $h_{[0,g]}(\theta)=\bigl(\operatorname{Re}e^{-i\theta}g\bigr)_+$, where $x_+=\max\{x,0\}$,
\begin{equation}\label{e:hZ}
h_{\sum_m[0,g_m]}(\theta)=\sum_m\bigl(\operatorname{Re}e^{-i\theta}g_m\bigr)_+ .
\end{equation}

We now show that $S_n$ is a zonotope.
Let
\begin{equation}\label{e:g_m}
g_m=\frac1n(1-\mu^m),\quad m=1,\dots,n-1.
\end{equation}
Then
\[
\lambda_k=\sum_{m=1}^{k}g_m,\qquad \lambda_k-\lambda_{k-1}=g_k,\qquad
g_{n-k}=\bar g_k,\qquad \sum_{m=1}^{n-1}g_m=1,\quad k=1\cdc n-1 .
\]
Since $1-e^{-i\varphi}=2\sin\tfrac\varphi2\,e^{i(\pi/2-\varphi/2)}$, we have
\begin{equation}\label{e:gen}
g_m=\frac2n\sin\frac{\pi m}{n}\;e^{\,i\bigl(\frac\pi2-\frac{\pi m}{n}\bigr)},
\end{equation}
so the $g_m$ are nonzero, $\operatorname{Re}g_m>0$, and for $n\ge3$, the slopes $\operatorname{Im}g_m/\operatorname{Re}g_m=\cot\frac{\pi m}{n}$ decrease strictly in~$m$.
Hence the polygonal line $0=\lambda_0,\lambda_1,\dots,\lambda_{n-1}=1$ is the graph of a strictly concave piecewise linear function $f_n$ on $[0,1]$ with breakpoints $\lambda_1,\dots,\lambda_{n-2}$ and $f_n(0)=f_n(1)=0$.
For $n=2$, we have $S_2=[0,1]$ and set $f_2\equiv0$. Thus, for every $n\ge2$,
\[
S_n=\{x+iy:\ 0\le x\le1,\ |y|\le f_n(x)\},
\]
and all $2(n-1)$ points $0,\lambda_1,\dots,\lambda_{n-2},1,\bar\lambda_{n-2},\dots,\bar\lambda_1$
are vertices of~$S_n$. Moreover,
$1-\lambda_k=\sum_{m>k}g_m=\overline{\lambda_{n-1-k}}$, so $S_n$ is centrally symmetric about~$\frac12$.

\begin{proposition}\label{p:zono}
$S_n$ is the zonotope with generators $g_1,\dots,g_{n-1}$:
\[
S_n=\sum_{m=1}^{n-1}[0,g_m]=\Bigl\{\sum_{m=1}^{n-1}t_m g_m:\ 0\le t_m\le 1\Bigr\}.
\]
\end{proposition}

\begin{pf}
For $n=2$ the assertion is immediate. Assume $n\ge3$.
By the above, $S_n$ is a centrally symmetric convex polygon, hence a zonogon, i.e., a two-dimensional zonotope~\cite[p.~319]{BoltyanskiMartiniSoltan97}.
The edges of a zonogon are translates of its generators, each generator contributing a pair of opposite parallel edges. The upper boundary of $S_n$ is the chain $0,\lambda_1,\dots,\lambda_{n-1}=1$ with edge vectors $g_1,\dots,g_{n-1}$, which are pairwise nonparallel; hence $S_n$ is a translate of
$\sum_{m=1}^{n-1}[0,g_m]$. Both sets have $0$ as their leftmost vertex (the point minimizing $\operatorname{Re}z$), so the translation is trivial.
\qed\end{pf}

The generators of $S_n$ are the edge vectors of either boundary chain joining two opposite vertices (see Fig.~\ref{f:counter}).

By~\eqref{e:hZ}, Proposition~\ref{p:zono} gives
\begin{equation}\label{e:hS}
h_{S_n}(\theta)
=\frac1n\sum_{m=1}^{n-1}\Bigl(\cos\theta-\cos\bigl(\theta+\tfrac{2\pi m}{n}\bigr)\Bigr)_+ .
\end{equation}

Moreover, Proposition~\ref{p:zono} gives a parametrization of the circulants that realize the points of $S_n$. The inclusion $S_n\subseteq\Lambda_n$ itself is \cite[Theorem~6]{AgaChe05LAA}, whose proof already uses convex combinations of circulants; here each point of $S_n$ is realized by a circulant whose weights are its zonotope coordinates, which will be used in Section~\ref{s:chords}.

\begin{corollary}\label{c:circ}
Let $z=\sum_{m=1}^{n-1}t_mg_m$ with $0\le t_m\le1$. Then $z$ is the eigenvalue at $(1,\mu,\dots,\mu^{n-1})^{\mathsf T}$ of the circulant
$\LT_t=\frac1n\sum_{m=1}^{n-1}t_m(I-Q^m)\in\cL_n$.
\end{corollary}

\begin{pf}
By Proposition~\ref{p:zono}, every $z\in S_n$ can be written as
\[
z=\sum_{m=1}^{n-1}t_m g_m,
\qquad 0\le t_m\le1.
\]
Let $Q$ be the cyclic permutation matrix with $q_{k,k+1}=q_{n,1}=1$, and set
\[
\LT_t=\frac1n\sum_{m=1}^{n-1}t_m(I-Q^m).
\]
Since the matrices $Q^m$, $m=1,\ldots,n-1$, have disjoint off-diagonal supports, $\LT_t\in\cL_n$. Moreover, $(1,\mu,\ldots,\mu^{n-1})^{\mathsf{T}}$ is an eigenvector of $\LT_t$ with eigenvalue
$
\sum_{m=1}^{n-1}t_m g_m=z.
$\qed
\end{pf}

\begin{rmk}\label{r:sym}
Since the matrices in $\cL_n$ are real, $\Lambda_n=\overline{\Lambda_n}$.
By~\cite[Theorem~1]{AgaChe05LAA}, $\lambda\in\Lambda_n\setminus\{0,1\}$ implies $1-\lambda\in\Lambda_n$; since $0,1\in\Lambda_n$, we have $\Lambda_n=1-\Lambda_n$.
The same symmetries hold for $S_n$: $S_n=\overline{S_n}$ by definition, in addition to the aforementioned central symmetry about~$\frac12$.
Hence if $\lambda\in\Lambda_n\setminus S_n$, then $\lambda$ has three companions in $\Lambda_n\setminus S_n$: $\bar\lambda$, $1-\lambda$, and $1-\bar\lambda$.
\end{rmk}

\section{An eigenvector description of $\Lambda_n$}\label{s:eigvec}

For $v\in\C^n$ and $i\in\supp v=\{i:v_i\ne0\}$, define the zonotope
\[
Z_i(v)=\frac1n\sum_{j\ne i}\Bigl[0,\;1-\frac{v_j}{v_i}\Bigr]\subset\C .
\]

\begin{proposition}\label{p:union}
Let $v\in\C^n\setminus\{0\}$ and $\lambda\in\C$. Then

$\LT v=\lambda v$ for some $\LT\in\cL_n$ $\iff$ $\lambda\in Z_i(v)$ for all $i\in\supp v$.
\end{proposition}
\begin{pf}
Row $i$ of $\LT v=\lambda v$ reads $\lambda v_i=\sum_{j\ne i}a_{ij}(v_i-v_j)$ with $a_{ij}\in\bigl[0,\frac1n\bigr]$.
For $v_i\ne0$, dividing by $v_i$ gives exactly $\lambda\in Z_i(v)$. For $v_i=0$, row $i$ can be taken to be zero. The rows are independent of one another.
\qed\end{pf}

Taking the union over all nonzero $v$ yields~\eqref{e:union}.

For $v=(1,\mu,\dots,\mu^{n-1})$, or any vector whose components are the $n$ distinct $n$th roots of unity (the components form a regular $n$-gon), one has $\{v_j/v_i: j\ne i\}=\{\mu^m: m=1,\dots,n-1\}$, so $Z_i(v)=\sum_m[0,g_m]=S_n$ for every~$i$.
Thus Conjecture~1 of~\cite{AgaChe05LAA} can be reformulated as saying that this vector is extremal: for every~$v$,
\[
\bigcap_{i\in\supp v} Z_i(v)\subseteq S_n.
\]
\nb{Sections~\ref{s:counter} and~\ref{s:chords} show that this fails for every $n\ge5$.}

\begin{rmk}
The condition $\lambda\in Z_i(v)$ is the Laplacian form of the condition $(1-\lambda)v_i\in F_i(v)$ underlying the geometric approach of Dmitriev and Dynkin~\cite{DmitrievDynkin46E}, where $F_i(v)=\bigl\{\sum_j p_jv_j:\ 0\le p_j\le\frac1n\ (j\ne i),\ \sum_jp_j=1\bigr\}.$
Indeed $(1-\lambda)v_i=v_i-\sum_{j\ne i}a_{ij}(v_i-v_j)$.
\end{rmk}

\section{An unweighted duplicated-cycle construction}\label{s:counter}

For $n\ge4$ let $\Gamma_n$ be the digraph on vertices $0,0',1,2,\dots,n-2$ with arcs
\[
0\to1,\quad 0'\to1,\quad 1\to2,\ \dots,\ n-3\to n-2,\quad n-2\to0,\quad n-2\to0' .
\]
This is the directed $(n-1)$-cycle in which vertex $0$ has been duplicated: $0'$ is a twin of $0$ with the same out-neighbor and in-neighbor.
Let $\LT_n=\LT(\Gamma_n)$ be the standardized Laplacian of the unweighted digraph~$\Gamma_n$.

\begin{lemma}\label{l:charpoly}
With $y=n\lambda$,
\begin{equation}\label{e:charLn}
\nvv{\det(\lambda I-\LT_n)
=
\frac{(-1)^{n-1}}{n^{\,n}}\,
(y-1)
\Bigl((1-y)^{n-2}(2-y)-2\Bigr).}
\end{equation}
In particular, for $n=7$,
\begin{equation}\label{e:charL7}
\nvv{\det(\lambda I-\LT_7)
=
\frac1{7^7}\,
y(y-1)
\bigl(y^5-7y^4+20y^3-30y^2+25y-11\bigr).}
\end{equation}
\end{lemma}

\begin{pf}
We index the coordinates of vectors in $\C^n$ by the vertices of $\Gamma_n$; let $e_k$ be the basis vector with $1$ at vertex $k$ and $0$ elsewhere.
The vector $u=e_0-e_{0'}$ satisfies
\[
\LT_nu=\frac1n u.
\]
The subspace $W=\{x:x_0=x_{0'}\}$ is invariant, since $\bigl(\LT_nx\bigr)_0-\bigl(\LT_nx\bigr)_{0'}=\frac1n(x_0-x_{0'}).$
Moreover,
$\C^n=\operatorname{span}\{u\}\oplus W.$

Identifying $x_0=x_{0'}$, the restriction of $\LT_n$ to $W$ is the weighted $(n-1)$-cycle Laplacian $D(I-C),$ where $C$ is the cyclic permutation matrix, and $D=\frac1n\operatorname{diag}(1,\dots,1,2).$

For a weighted directed cycle with $D=\operatorname{diag}(d_1,\dots,d_m)$, a direct expansion gives
\[
\det\bigl(\lambda I-D(I-C)\bigr)
=
(-1)^m\prod_{i=1}^m d_i
\left[
\prod_{i=1}^m\left(1-\frac{\lambda}{d_i}\right)-1
\right].
\]
Here $m=n-1$, $\prod_i d_i=\frac{2}{n^{\,n-1}},$
and, with $y=n\lambda$,
$
\prod_i\left(1-\frac{\lambda}{d_i}\right)
=
(1-y)^{n-2}\left(1-\frac y2\right).
$

Therefore
\[
\det\Bigl(\lambda I-\LT_n\bigr|_W\Bigr)
=
\frac{(-1)^{n-1}}{n^{\,n-1}}
\Bigl((1-y)^{n-2}(2-y)-2\Bigr).
\]

\nvv{Multiplying by the remaining factor $\lambda-\frac1n=\frac{y-1}n$, corresponding to the eigenvector $u$, gives~\eqref{e:charLn}.}

\nvv{For $n=7$,
\[
(1-y)^5(2-y)-2
=
y\bigl(y^5-7y^4+20y^3-30y^2+25y-11\bigr),
\]
which gives~\eqref{e:charL7}.}
\qed\end{pf}

The eigenvectors with $x_0=x_{0'}$ are therefore those of
the quotient, a weighted directed $(n-1)$-cycle with $D=\frac1n\operatorname{diag}(1,\dots,1,2).$
Thus the quotient belongs to the class of weighted cycle Laplacians studied by Hariprasad~\cite{Hariprasad026}, with one
weight equal to $2/n$ and all remaining weights equal to~$1/n$. The plain $(n-1)$-cycle with weights $\frac1n$ (plus an isolated vertex) has the eigenvalue $\frac1n\bigl(1-e^{-2\pi i/(n-1)}\bigr)$, which lies strictly inside $S_n$ just below the edge $[\lambda_1,\lambda_2]$. For $n=7$, adding the twin increases one weight and pushes the corresponding eigenvalue across the edge.

\begin{theorem}\label{t:counter}
$\LT_7$ has an eigenvalue
\[
\lambda^\ast
=
0.0813812844676782\ldots
+
0.1339660536359618\ldots\,i
\]
lying beyond the edge $[\lambda_1,\lambda_2]$ of $S_7$ \textup{(}see Fig.~\textup{\ref{f:counter})}, at distance $2.1130622\ldots\cdot10^{-4}$ from~$S_7$.
In particular, $\sigma\bigl(\LT_7\bigr)\not\subseteq S_7$ and $S_7\subsetneq\Lambda_7$.
\end{theorem}

\begin{pf}
By Lemma~\ref{l:charpoly}, apart from the eigenvalue $1/7$, the nonzero eigenvalues of $\LT_7$ are obtained from the zeros of
\[
p(y)=y^5-7y^4+20y^3-30y^2+25y-11,
\qquad y=7\lambda.
\]

Let
\[
c=
0.5696689912737478
+
0.9377623754517331\,i .
\]

The certification in \ref{a:certification} gives a zero $y^\ast$ of $p$ satisfying $|y^\ast-c|<10^{-16}.$

Let $\mu=e^{-2\pi i/7}$. In the coordinate $y=7\lambda$, the signed Euclidean distance from the line through the edge
$[\lambda_1,\lambda_2]$, positive on the exterior side, is
\[
d_7(y)
=
\frac17
\left(
\operatorname{Re}\bigl(\bar\mu(1-y)\bigr)-1
\right).
\]
The same appendix shows that
\[
d_7(y^\ast)=2.1130622\ldots\cdot10^{-4}>0
\]
and that the orthogonal projection lies in the interior of $[\lambda_1,\lambda_2]$. Hence $\lambda^\ast=y^\ast/7$ lies outside $S_7$, at the stated distance.
\qed\end{pf}

\begin{figure}[pos=htb] 
\centering
\includegraphics[width=\textwidth]{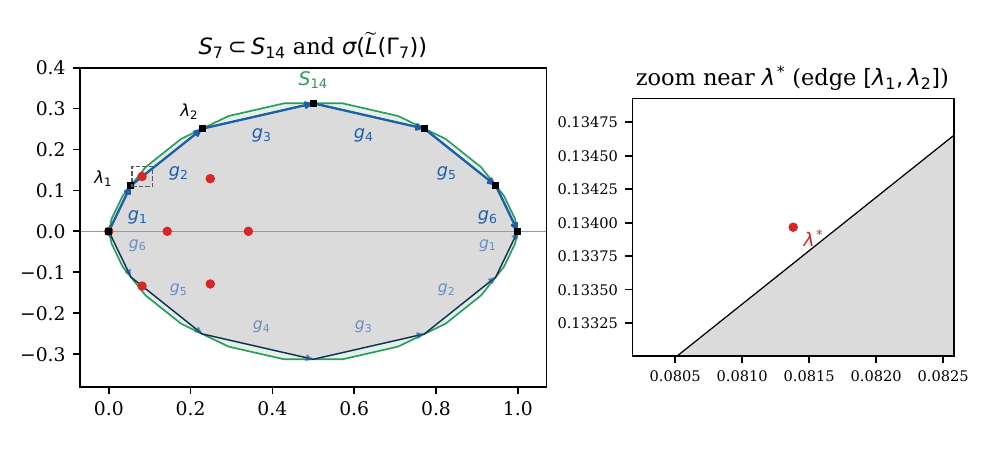}
\caption{Left: $S_7$ (grey) inside $S_{14}$ (green outline). The arrows show the generators $g_1,\dots,g_6$ of $S_7$: its upper boundary is the chain
$g_1,\dots,g_6$ from $0$ to $1$, with vertices $\lambda_k=g_1+\dots+g_k$, and its lower boundary is the chain $g_6,\dots,g_1$. The vertices of $S_7$ lie on edges of $S_{14}$ (Proposition~\ref{p:nested}). Red dots: the spectrum of $\LT(\Gamma_7)$.
Right: zoom near $\lambda^\ast$; the eigenvalue lies outside the edge $[\lambda_1,\lambda_2]=\lambda_1+[0,g_2]$ of~$S_7$.}
\label{f:counter}
\end{figure}

\subsection{Other orders and asymptotics}

To examine the same construction in other orders, for each $n$ we numerically select the root $y_n^\ast$ of
\[
(1-y)^{n-2}(2-y)=2
\]
near $1-e^{-2\pi i/(n-1)}$, and put $\lambda_n^\ast=y_n^\ast/n$.
Let $\delta_n$ be its signed distance from the line through $[\lambda_1,\lambda_2]$, positive on the exterior side of $S_n$. The large-$n$ branch is characterized intrinsically in the theorem below.

\begin{table}[pos=ht]
\centering
\caption{The eigenvalue branch $\lambda_n^\ast$, its signed distance
$\delta_n$ from the supporting line of $[\lambda_1,\lambda_2]$, and
the scaled quantity $n^4\delta_n$.}
\label{tab:delta}
\begin{tabular}{rllll}
\toprule
$n$ & $\operatorname{Re}\lambda^\ast_n$ & $\operatorname{Im}\lambda^\ast_n$ & $\delta_n$ & $n^4\delta_n$\\
\midrule
6 & 0.13503588 & 0.17386416 & $-2.805\cdot10^{-4}$ & $-0.364$ \\
7 & 0.081381284 & 0.13396605 & $\phantom{-}2.113\cdot10^{-4}$ & 0.507 \\
8 & 0.052574304 & 0.10475785 & $\phantom{-}2.877\cdot10^{-4}$ & 1.178 \\
9 & 0.035832644 & 0.083545162 & $\phantom{-}2.573\cdot10^{-4}$ & 1.688 \\
10 & 0.025476421 & 0.067911669 & $\phantom{-}2.083\cdot10^{-4}$ & 2.083 \\
12 & 0.014178543 & 0.047142501 & $\phantom{-}1.277\cdot10^{-4}$ & 2.648 \\
20 & 0.0028191324 & 0.016669868 & $\phantom{-}2.294\cdot10^{-5}$ & 3.671 \\
50 & 0.00016673942 & 0.002583874 & $\phantom{-}7.147\cdot10^{-7}$ & 4.467 \\
100 & $2.0286542\cdot10^{-5}$ & 0.00063745761 & $\phantom{-}4.707\cdot10^{-8}$ & 4.707 \\
300 & $7.3780222\cdot10^{-7}$ & $7.0158822\cdot10^{-5}$ & $\phantom{-}6.0\cdot10^{-10}$ & 4.86 \\
\bottomrule
\end{tabular}
\end{table}

\begin{theorem}[asymptotics of the violation]\label{t:asymptotic}
For every sufficiently large $n$, the equation $(1-y)^{n-2}(2-y)=2$ has exactly
one root $y_n^\ast$ such that $1-y_n^\ast=e^{w_n}$ with $|w_n+2\pi i/n|<\pi/n$. Moreover, $w_n=-2\pi i/n+O\bigl(n^{-2}\bigr)$, and the signed distance $\delta_n$ of\, $\lambda_n^\ast=y_n^\ast/n$\, from the line through the edge $[\lambda_1,\lambda_2]$, positive on the exterior side, satisfies
\[
\delta_n=\frac{\pi^2}{2n^4}-\frac{9\pi^2}{4n^5}+O\bigl(n^{-6}\bigr).
\]
In particular, $\delta_n>0$ and $S_n\subsetneq\Lambda_n$\, for all sufficiently large~$n$.
\end{theorem}

\begin{pf}
Put $q_n=1-y_n^\ast$. Then $q_n^{\,n-2}(1+q_n)=2.$

Writing $t=1/n$ and $q_n=e^w$, and using $1+e^w=2e^{w/2}\cosh\frac w2,$ we obtain
\[
e^{\bigl(n-\frac32\bigr)w}\cosh\frac w2=1.
\]

If $1-y=e^{w}$ is a root with $|w+2\pi i/n|<\pi/n$, then $\bigl(n-\frac32\bigr)w+\log\cosh\frac w2=2\pi i k$ for some integer~$k$, and
\[
2\pi|k+1|
\le\bigl|nw+2\pi i\bigr|+\tfrac32|w|+\bigl|\log\cosh\tfrac w2\bigr|
<\pi+O\bigl(n^{-1}\bigr),
\]
so $k=-1$ for large~$n$. Thus such roots are exactly the solutions of
\[
\left(n-\frac32\right)w+\log\cosh\frac w2=-2\pi i
\]
in this disk.
After multiplication by $t$ this becomes
\[
F(t,w):=
\left(1-\frac32t\right)w
+
t\log\cosh\frac w2
+
2\pi i t
=0.
\]
Since $F(0,0)=0,$\, $\frac{\partial F}{\partial w}(0,0)=1,$ the analytic implicit-function theorem gives a unique analytic solution $w=w(t)$ near $t=0$.

Now
\[
\log\cosh\frac w2
=
\frac{w^2}{8}+O\bigl(w^4\bigr).
\]

Substituting\, $w(t)
=
a_1t+a_2t^2+a_3t^3+a_4t^4+O\bigl(t^5\bigr)$\, into\, $F(t,w(t))=0$\, and comparing coefficients gives
\[
a_1=-2\pi i,\qquad
a_2=-3\pi i,\qquad
a_3=\frac{\pi^2}{2}-\frac{9\pi i}{2},\qquad
a_4=\frac{9\pi^2}{4}-\frac{27\pi i}{4}.
\]

Hence
\[
w(t)
=
-2\pi i t
-3\pi i t^2
+
\left(
\frac{\pi^2}{2}-\frac{9\pi i}{2}
\right)t^3
+
\left(
\frac{9\pi^2}{4}-\frac{27\pi i}{4}
\right)t^4
+
O\bigl(t^5\bigr).
\]
In particular, $w(t)=-2\pi i t+O(t^2)$, so the root $1-y=e^{w(1/n)}$ lies in the disk $|w+2\pi i/n|<\pi/n$ for large~$n$; conversely, every solution in this disk satisfies $|w|<3\pi/n$ and hence, for large~$n$, lies in the neighborhood where $w(t)$ is the unique solution. This proves the first two assertions of the theorem.

We next express the signed distance from the edge. Let $\mu=e^{-2\pi i/n}.$
Since
\[
\lambda_1=\frac{1-\mu}{n},
\qquad
\lambda_2-\lambda_1=\frac{1-\mu^2}{n},
\]
the functional $\operatorname{Re}\bigl(-\bar\mu(\lambda-\lambda_1)\bigr)$ is signed Euclidean distance from the line through $[\lambda_1,\lambda_2]$, with the positive side exterior to $S_n$.
Using\, $\lambda_n^\ast=(1-q_n)/n,$\, we therefore obtain
\[
\delta_n
=
\frac1n
\left(
\operatorname{Re}\bigl(e^{2\pi i/n}q_n\bigr)-1
\right).
\]

Thus, with $t=1/n$,
\[
\delta_n
=
t\left(
\operatorname{Re}e^{\,w(t)+2\pi i t}-1
\right).
\]

Set $
u(t)=w(t)+2\pi i t.
$
Then
\[
u(t)
=
-3\pi i t^2
+
\left(
\frac{\pi^2}{2}-\frac{9\pi i}{2}
\right)t^3
+
\left(
\frac{9\pi^2}{4}-\frac{27\pi i}{4}
\right)t^4
+
O\bigl(t^5\bigr).
\]

Since $u(t)=O\bigl(t^2\bigr)$,
\[
\operatorname{Re}e^{u(t)}-1
=
\operatorname{Re}u(t)
+
\frac12\operatorname{Re}\bigl(u(t)^2\bigr)
+
O\bigl(t^5\bigr).
\]

Here
\[
\operatorname{Re}u(t)
=
\frac{\pi^2}{2}t^3
+
\frac{9\pi^2}{4}t^4
+
O\bigl(t^5\bigr),
\]
while
$
u(t)^2
=
-9\pi^2t^4+O\bigl(t^5\bigr).
$
Therefore
\[
\operatorname{Re}e^{u(t)}-1
=
\frac{\pi^2}{2}t^3
-
\frac{9\pi^2}{4}t^4
+
O\bigl(t^5\bigr),
\]
and hence
\[
\delta_n
=
\frac{\pi^2}{2}t^4
-
\frac{9\pi^2}{4}t^5
+
O\bigl(t^6\bigr),\quad \text{where\;\,}t=1/n.
\]

Since $\lambda_n^\ast$ is an eigenvalue of $\LT_n$ by Lemma~\ref{l:charpoly}, the positivity of $\delta_n$ implies $S_n\subsetneq\Lambda_n$ for all
sufficiently large~$n$.
\qed\end{pf}

Theorem~\ref{t:asymptotic} explains the last column of Table~\ref{tab:delta}, since
\[
n^4\delta_n\longrightarrow\frac{\pi^2}{2}=4.934802\ldots .
\]
A direct computation in 50-digit arithmetic gives $\delta_n>0$ for every integer $7\le n\le300$, suggesting the stronger statement $\delta_n>0$ for every $n\ge7$.

\begin{rmk}
Continuing the expansion by one more order gives
\[
\delta_n
=
\frac{\pi^2}{2n^4}
-
\frac{9\pi^2}{4n^5}
+
\left(
\frac{\pi^4}{12}
-
\frac{27\pi^2}{4}
\right)\frac1{n^6}
+
O\bigl(n^{-7}\bigr).
\]
The relatively large negative correction terms explain why the
leading asymptotic term is still a poor numerical approximation at
$n=7$.
\end{rmk}

\begin{nbblock}
\section{Hamiltonian cycles with chords: counterexamples for all $n\ge5$}\label{s:chords}

In this section the vertices are $1,\dots,n$, indices are taken modulo~$n$, $n\ge3$,
and $Q$ is the cyclic permutation matrix from the proof of Corollary~\ref{c:circ}.
Put $\alpha=2\pi/n$, so that $\mu=e^{-i\alpha}$.
For $\omega=(\omega_1,\dots,\omega_n)\in[0,1]^n$ and $D_\omega=\operatorname{diag}(\omega_1,\dots,\omega_n)$, let
\begin{equation}\label{e:Lw}
\LT(\omega)=\frac1n\Bigl((I-Q)+D_\omega\bigl(I-Q^2\bigr)\Bigr).
\end{equation}

Since $n\ge3$, the matrices $Q$ and $Q^2$ have disjoint off-diagonal supports, so
$\LT(\omega)\in\cL_n$: it is the standardized Laplacian of the weighted digraph formed by
the Hamiltonian cycle $1\to2\to\dots\to n\to1$ with arc weights~$1$ and the chords
$k\to k+2$ with weights $\omega_k$ (chords of weight $0$ being absent), with $b=1$.

For $\omega=0$, the vector $(1,\mu,\dots,\mu^{n-1})^{\mathsf T}$ is an eigenvector of
$\LT(0)=\frac1n(I-Q)$ with eigenvalue $g_1=\lambda_1$ ($g_m$ are the generators \eqref{e:g_m} of $S_n$), a vertex of~$S_n$.
If all $\omega_k$ are equal to some $\tau\in[0,1]$, then $\LT(\omega)$ is circulant and, by the proof
of Corollary~\ref{c:circ}, this eigenvalue becomes $\lambda_1+\tau g_2\in[\lambda_1,\lambda_2]$:
it moves along the edge $[\lambda_1,\lambda_2]$ and reaches $\lambda_2$ at $\tau=1$.
We show that suitable nonconstant weights can push it beyond the supporting line of this edge, the effect being of second order in the size of the weights.

The position of a point relative to the edge is measured by
\[
d(z)=\operatorname{Re}\bigl(-\bar\mu\,(z-\lambda_1)\bigr),\qquad z\in\C .
\]

\begin{lemma}\label{l:edge}
The line through $\lambda_1$ and $\lambda_2$ is the supporting line of $S_n$ with
outward normal~$-\mu$. Hence $d(z)$ is the signed distance from $z$ to this line,
positive on the exterior side, and $d(z)>0$ implies $z\notin S_n$.
\end{lemma}

\begin{pf}
Let $e^{i\theta}=-\mu$. Since $-\bar\mu g_m=\frac1n\bigl(\mu^{m-1}-\bar\mu\bigr)$,
\[
\operatorname{Re}\bigl(e^{-i\theta}g_m\bigr)=\frac1n\bigl(\cos(m-1)\alpha-\cos\alpha\bigr),
\]
which is positive for $m=1$, zero for $m=2$, and negative for $3\le m\le n-1$, because then
$2\le m-1\le n-2$. By~\eqref{e:hZ} and Proposition~\ref{p:zono},
$h_{S_n}(\theta)=\operatorname{Re}(e^{-i\theta}g_1)=\operatorname{Re}(e^{-i\theta}\lambda_1)$,
and the supporting line $\{z:\operatorname{Re}(e^{-i\theta}z)=h_{S_n}(\theta)\}$ contains
$\lambda_1$ and $\lambda_2=\lambda_1+g_2$.
\qed\end{pf}

For $\omega\in\R^n$ put
\[
\hat\omega(j)=\frac1n\sum_{k=1}^n \omega_k\,\mu^{j(k-1)},\qquad
\bar\omega=\hat\omega(0),\qquad
\gamma(\omega)=\frac1n\sum_{k=1}^n \omega_k\omega_{k+1}-\bar\omega^{\,2}.
\]
Thus $\gamma(\omega)$ is the cyclic autocovariance of the weight sequence at lag~$1$.
Since $\omega$ is real, $\hat\omega(-j)=\overline{\hat\omega(j)}$.

\begin{lemma}\label{l:pert}
Let $\omega\in[0,1]^n$. There exist $\varepsilon>0$ and a function $\lambda(t)$, analytic for
$|t|<\varepsilon$, such that $\lambda(0)=\lambda_1$ and $\lambda(t)$ is an eigenvalue of\,
$\LT(t\omega)$\, for\, $0\le t<\varepsilon$. Moreover,
\begin{equation}\label{e:pert}
\lambda(t)=\lambda_1+t\,\bar\omega\,g_2+t^2\beta+O\bigl(t^3\bigr),
\qquad
d\bigl(\lambda(t)\bigr)=\frac2n\sin^2\!\frac{2\pi}{n}\;\gamma(\omega)\,t^2+O\bigl(t^3\bigr),
\end{equation}
where
\begin{equation}\label{e:beta}
\operatorname{Re}\bigl(-\bar\mu\beta\bigr)
=\frac2n\sin^2\!\alpha\sum_{j=1}^{n-1}\cos(j\alpha)\,|\hat\omega(j)|^2
=\frac2n\sin^2\!\alpha\;\gamma(\omega).
\end{equation}
\end{lemma}

\begin{pf}
For $m=0,\dots,n-1$ let $f_m=(\mu^{m(k-1)})_{k=1}^n$. Then $Qf_m=\mu^mf_m$, and the matrix
$F=[f_0,\dots,f_{n-1}]$ satisfies $F^{-1}=\frac1nF^{*}$. Since
$(D_\omega f_m)_k=\omega_k\mu^{m(k-1)}$, we have $D_\omega f_m=\sum_{l}\hat\omega(m-l)f_l$. Hence
\[
F^{-1}\LT(t\omega)F=R+tE,\qquad R=\operatorname{diag}(\kappa_0,\dots,\kappa_{n-1}),\quad
\kappa_m=\frac{1-\mu^m}{n},\quad
E=(e_{lm}),\quad e_{lm}=\frac{1-\mu^{2m}}{n}\,\hat\omega(m-l).
\]
The numbers $\kappa_m$ are pairwise distinct and $\kappa_1=g_1=\lambda_1$. Write $R'$ and
$E'$ for the principal submatrices of $R$ and $E$ indexed by $\{0,2,3,\dots,n-1\}$, and
$e_{1\ast}$, $e_{\ast1}$ for the rest of row and column~$1$ of~$E$ in these positions. By the
Schur complement formula, whenever $R'+tE'-\lambda I$ is invertible,
\[
\det(R+tE-\lambda I)=\det(R'+tE'-\lambda I)\,\varphi(\lambda,t),
\quad
\varphi(\lambda,t)=\kappa_1+te_{11}-\lambda-t^2e_{1\ast}\bigl(R'+tE'-\lambda I\bigr)^{-1}e_{\ast1}.
\]
The function $\varphi$ is analytic near $(\kappa_1,0)$, $\varphi(\kappa_1,0)=0$, and
$\partial\varphi/\partial\lambda(\kappa_1,0)=-1$. By the implicit function theorem there
is an analytic $\lambda(t)$ with $\lambda(0)=\kappa_1$ and $\varphi(\lambda(t),t)=0$; since
$\kappa_1$ is not an eigenvalue of $R'$, the first factor does not vanish near
$(\kappa_1,0)$, so $\lambda(t)$ is an eigenvalue of $\LT(t\omega)$. Expanding $\varphi$ gives
\[
\lambda(t)=\kappa_1+t\,e_{11}+t^2\sum_{m\ne1}\frac{e_{1m}e_{m1}}{\kappa_1-\kappa_m}+O(t^3).
\]

Here $e_{11}=\frac1n(1-\mu^2)\hat\omega(0)=\bar\omega\,g_2$. For $m\ne1$ put $j=m-1$, so that $j$
runs over $1,\dots,n-1$ modulo~$n$. Then
$e_{1m}e_{m1}=\frac1{n^2}(1-\mu^2)(1-\mu^{2j+2})\,|\hat\omega(j)|^2$ and
$\kappa_1-\kappa_m=-\frac1n\mu\,(1-\mu^{j})$. Therefore
\[
-\bar\mu\beta=\frac1n\sum_{j=1}^{n-1}
\frac{\bar\mu^{\,2}(1-\mu^2)(1-\mu^{2j+2})}{1-\mu^{j}}\,|\hat\omega(j)|^2 .
\]
Using $1-e^{-i\varphi}=2i\sin\frac\varphi2\,e^{-i\varphi/2}$, we find that the $j$th fraction
equals $2i\,\dfrac{\sin\alpha\,\sin(j+1)\alpha}{\sin(j\alpha/2)}\,e^{-ij\alpha/2}$, whose
real part is $2\sin\alpha\,\sin(j+1)\alpha$. Since $|\hat\omega(n-j)|=|\hat\omega(j)|$, we may
replace $\sin(j+1)\alpha$ by
$\frac12\bigl(\sin(j+1)\alpha+\sin(1-j)\alpha\bigr)=\sin\alpha\cos j\alpha$, which gives
the first equality in~\eqref{e:beta}. For the second,
\[
\sum_{j=0}^{n-1}\mu^{j}\,|\hat\omega(j)|^2
=\frac1{n^2}\sum_{k,l}\omega_k\omega_l\sum_{j=0}^{n-1}\mu^{j(k-l+1)}
=\frac1n\sum_{k=1}^n \omega_k\omega_{k+1},
\]
and taking real parts and subtracting the term $j=0$, which is $\bar\omega^{\,2}$, yields
$\sum_{j=1}^{n-1}\cos(j\alpha)|\hat\omega(j)|^2=\gamma(\omega)$.

Finally, $d(\lambda_1+u)=\operatorname{Re}(-\bar\mu u)$ is real linear in $u$ and vanishes
for $u=\bar\omega\,g_2$ by Lemma~\ref{l:edge}, so $d(\lambda(t))=t^2\operatorname{Re}(-\bar\mu\beta)+O(t^3)$.
\qed\end{pf}

\begin{rmk}\label{r:gamma}
Only the harmonics $j$ with $\cos(2\pi j/n)>0$, i.e., $\min(j,n-j)<n/4$, push the
eigenvalue outward. For $n\le4$ there is no such harmonic with $j\ne0$. In fact,
$\gamma(\omega)=-\frac1{18}\sum_k(\omega_k-\omega_{k+1})^2$ for $n=3$, and
$\gamma(\omega)=-\frac1{16}(\omega_1-\omega_2+\omega_3-\omega_4)^2$ for $n=4$. Thus for $n\le4$ the second-order term
never points outward. This is consistent with $\Lambda_3=S_3$~\cite{AgaChe05LAA}.
For $n\ge5$, by contrast, $\gamma(\omega)>0$ for every $\omega$ that is a nonzero multiple of
$e_k+e_{k+1}$, i.e., for two consecutive chords of equal weight.
\end{rmk}

For $n\ge4$ and $0\le t\le1$, let $H_n(t)$ be the weighted digraph formed by the Hamiltonian
cycle $1\to2\to\dots\to n\to1$ with arc weights~$1$ and the two chords $1\to3$ and $2\to4$
of weight~$t$. Its standardized Laplacian is $\LT(t(e_1+e_2))$.

\begin{theorem}\label{t:allfive}
Let $n\ge5$. The eigenvalue $\lambda(t)$ of $\LT(H_n(t))$ with $\lambda(0)=\lambda_1$
\textup{(}Lemma~\textup{\ref{l:pert})} satisfies
\[
d\bigl(\lambda(t)\bigr)=\frac{2(n-4)}{n^3}\,\sin^2\!\frac{2\pi}{n}\;t^2+O(t^3),
\qquad t\to0.
\]
Consequently, $\lambda(t)\notin S_n$ for all sufficiently small $t>0$, and
$S_n\subsetneq\Lambda_n$ for every $n\ge5$.
\end{theorem}

\begin{pf}
For $\omega=e_1+e_2$ we have $\bar\omega=2/n$ and $\sum_k\omega_k\omega_{k+1}=1$, so
$\gamma(\omega)=\frac1n-\frac4{n^2}=\frac{n-4}{n^2}>0$. The claim follows from
Lemma~\ref{l:pert} and Lemma~\ref{l:edge}.
\qed\end{pf}

By Remark~\ref{r:sym}, the points $\bar\lambda(t)$, $1-\lambda(t)$, and $1-\bar\lambda(t)$
also belong to $\Lambda_n\setminus S_n$; they lie beyond the edges $\bigl[\bar\lambda_1,\bar\lambda_2\bigr]$,
$\bigl[\bar\lambda_{n-3},\bar\lambda_{n-2}\bigr]$, and $\bigl[\lambda_{n-3},\lambda_{n-2}\bigr]$, respectively.
The same applies to Theorems~\ref{t:counter} and~\ref{t:five}.

Together with $\Lambda_n=S_n$ for $n\le3$~\cite{AgaChe05LAA}, Theorem~\ref{t:allfive} leaves $n=4$ as the only order for which the conjecture is undecided; \nvv{this case has recently been settled in the affirmative~\cite{Liu026v3,Marshall026}.}

Theorem~\ref{t:allfive} is purely existential in~$t$. The following example, certified in
\ref{a:cert5}, shows that the effect is not small.

\begin{theorem}\label{t:five}
With $x=10\lambda$,
\[
\det\Bigl(\lambda I-\LT\bigl(H_5\bigl(\tfrac12\bigr)\bigr)\Bigr)
=10^{-5}\,x\,\bigl(x^4-12x^3+57x^2-134x+140\bigr),
\]
and $\LT\bigl(H_5\bigl(\frac12\bigr)\bigr)$ has the eigenvalue
\[
\lambda^\sharp=0.2006968422057107\ldots+0.2121354840243571\ldots\,i
\]
lying beyond the edge $[\lambda_1,\lambda_2]$ of $S_5$, at distance
$d(\lambda^\sharp)=1.5374983\ldots\cdot10^{-3}$ from~$S_5$. In particular,
$S_5\subsetneq\Lambda_5$.
\end{theorem}

\begin{pf}
Since $\LT\bigl(H_5\bigl(\frac12\bigr)\bigr)=\frac15L\bigl(H_5\bigl(\frac12\bigr)\bigr)$, the matrix
$10\,\LT\bigl(H_5\bigl(\frac12\bigr)\bigr)=2L\bigl(H_5\bigl(\frac12\bigr)\bigr)$ has integer entries, and its
characteristic polynomial $x\bigl(x^4-12x^3+57x^2-134x+140\bigr)$ is obtained by direct expansion. The enclosure of its zero, the bound for $d$, and the position of the foot of the
perpendicular are certified in~\ref{a:cert5}.
\qed\end{pf}

The distance in Theorem~\ref{t:five} is about $7.3$ times the distance in
Theorem~\ref{t:counter}. Optimizing the chord weights in~\eqref{e:Lw} numerically increases it
for $n=5$ to about $2.5\cdot10^{-3}$, attained at
$\omega\approx(0.731,\,0.344,\,0,\,0,\,0.344)$. \nH{Liu~\nvv{\cite[Theorem~5]{Liu026v3}} gives explicit digraphs of this kind with integer weights: the Hamiltonian cycle with arc weights~$12$ and the chords $0\to2$, $1\to3$, $2\to4$ of weights $4$, $9$, $4$; for $n=5$ and $n=6$, the eigenvalue lies at distances $2.466\ldots\cdot10^{-3}$ and $1.791\ldots\cdot10^{-3}$ beyond $[\lambda_1,\lambda_2]$, certified in ball arithmetic.}

\begin{rmk}[relation to $\Gamma_n$]\label{r:path}
At $t=1$, the digraph $H_n(1)$ is unweighted, and the spectral construction of Section~\ref{s:counter} reappears.
Indeed, write $\LT=\LT(H_n(1))$. Then $n(\LT x)_1=2x_1-x_2-x_3$, $n(\LT x)_2=2x_2-x_3-x_4$, and $n(\LT x)_3=x_3-x_4$. Hence the
subspace $\{x:x_2=x_3\}$ is invariant, and identifying $x_2=x_3$ yields the weighted
$(n-1)$-cycle $1\to\{2,3\}\to4\to\dots\to n\to1$ with $D=\frac1n\operatorname{diag}(2,1,\dots,1)$.
This is the quotient that appears in the proof of Lemma~\ref{l:charpoly}. Accordingly,
\[
\det\bigl(\lambda I-\LT(H_n(1))\bigr)
\nvv{=\frac{(-1)^{n-1}}{n^{\,n}}\,(y-2)\Bigl((1-y)^{n-2}(2-y)-2\Bigr),}
\qquad y=n\lambda,
\]
so $\LT(H_n(1))$ and $\LT_n$ have the same eigenvalues except that $\frac1n$ is replaced
by~$\frac2n$. Numerically, for $5\le n\le20$, the branch $\lambda(t)$ of
Theorem~\ref{t:allfive}, continued along $0\le t\le1$, ends at the eigenvalue $\lambda_n^\ast$
of Section~\ref{s:counter}. For $7\le n\le20$, it stays outside $S_n$ for all
$0<t\le1$. For $n=5,6$, it leaves $S_5$, respectively $S_6$, and returns:
$d(\lambda(1))=\delta_5=-1.99\cdot10^{-3}$ and $\delta_6=-2.80\cdot10^{-4}$. Thus the counterexample $\Gamma_7$ is the endpoint of a path of
eigenvalues that starts at the vertex $\lambda_1$ tangentially to the edge
$[\lambda_1,\lambda_2]$ and leaves the line of this edge only quadratically.
\end{rmk}
\subsection{The other edges}\label{ss:edges}

The construction extends to the edge $[\lambda_m,\lambda_{m+1}]$ for every
$1\le m\le n-2$. Put $p=m+1$, $\nu_m=-e^{-i\pi p/n}$, and
\[
d_m(z)=\operatorname{Re}\bigl(\bar\nu_m(z-\lambda_m)\bigr),\qquad z\in\C ,
\]
so that $\nu_1=-\mu$ and $d_1=d$.

\begin{lemma}\label{l:edgem}
The line through $\lambda_m$ and $\lambda_{m+1}$ is the supporting line of $S_n$ with
outward normal $\nu_m$. Hence $d_m(z)>0$ implies $z\notin S_n$.
\end{lemma}

\begin{pf}
By~\eqref{e:gen}, $\arg g_s=\frac\pi2-\frac{\pi s}{n}$, so
$\operatorname{Re}(\bar\nu_m g_s)=|g_s|\sin\frac{\pi(p-s)}{n}$. This is positive for
$s\le m$, zero for $s=p$, and negative for $s\ge p+1$. By~\eqref{e:hZ} and
Proposition~\ref{p:zono}, $h_{S_n}$ in the direction $\nu_m$ equals
$\sum_{s\le m}\operatorname{Re}(\bar\nu_m g_s)=\operatorname{Re}(\bar\nu_m\lambda_m)$, and the
supporting line contains $\lambda_m$ and $\lambda_{m+1}=\lambda_m+g_p$.
\qed\end{pf}

For $\omega\in\R^n,$ let
\begin{equation}\label{e:Lmw}
\LT_m(\omega)=\frac1n\Bigl(\sum_{s=1}^{m}\bigl(I-Q^s\bigr)+D_\omega\bigl(I-Q^{p}\bigr)\Bigr).
\end{equation}
If $\omega\in[0,1]^n$, then $\LT_m(\omega)\in\cL_n$. Indeed, since $p\le n-1$, the matrices $Q,\dots,Q^{p}$ have disjoint off-diagonal supports. It is the standardized Laplacian of the digraph that contains all arcs $k\to k+s$, $1\le s\le m$, with weight~$1$, i.e., the union of the first $m$ powers of the Hamiltonian cycle, together with the chords $k\to k+p$ of weights~$\omega_k$. For $m=1$, this is~\eqref{e:Lw}. With $f_k$ as in the proof of Lemma~\ref{l:pert} and
\[
F_m(k)=\sum_{s=0}^{m}\mu^{sk},
\]
we have $\LT_m(0)f_k=\kappa_kf_k$, where $\kappa_k=\frac1n\bigl(p-F_m(k)\bigr)$, and
$\kappa_1=g_1+\dots+g_m=\lambda_m$. For constant weights $\omega_k=\tau$,
the eigenvalue at $f_1$ is $\lambda_m+\tau g_p\in[\lambda_m,\lambda_{m+1}]$.

\begin{lemma}\label{l:simple}
For $1\le m\le n-2$, $\lambda_m$ is a simple eigenvalue of $\LT_m(0)$.
\end{lemma}

\begin{pf}
Let $k\not\equiv1$, $\delta=\gcd(k,n)$, and $n'=n/\delta$. Note that
$\kappa_k=\frac1n\sum_{s=1}^m\bigl(1-\mu^{sk}\bigr)$.

First let $\delta=1$. Then the residues $r_s\equiv sk \pmod n$, $1\le s\le m$, are distinct and
nonzero, and $\kappa_k=\sum_{s=1}^m g_{r_s}$. Choose $\theta$ with
$e^{i\theta}=-e^{-i\pi(m+1/2)/n}$. As in the proof of Lemma~\ref{l:edgem},
$\operatorname{Re}(e^{-i\theta}g_r)$ is positive for $r\le m$ and negative for $r>m$. Hence
$\operatorname{Re}(e^{-i\theta}\kappa_k)\le\operatorname{Re}(e^{-i\theta}\lambda_m)$, with equality
only if $\{r_1,\dots,r_m\}=\{1,\dots,m\}$. The latter is impossible. Indeed, let
$A=\{1,\dots,m\}$ be invariant under multiplication by $k$ modulo $n$. Then $k\in A$, so
$2\le k\le m$. If $2m<n$, let $j$ be maximal with $jk\le m$. Then $j+1\le m$, while
$m<(j+1)k\le 2m<n$, so $(j+1)k\notin A$, a contradiction. If $2m\ge n$, then the complement
$\{m+1,\dots,n-1\}$ is also invariant, and hence so is $A'=\{1,\dots,n-1-m\}$, its image
under $r\mapsto n-r$. Since $2(n-1-m)<n$, the same argument applied to $A'$ gives a
contradiction.

Now let $\delta>1$. If $k\equiv0$, then $\kappa_k=0\ne\lambda_m$. Otherwise, $\mu^k$ is a
primitive $n'$th root of unity. Grouping the indices $s$ by the residue of
$s\,k/\delta$ modulo~$n'$, and noting that each residue occurs at most
$\lceil m/n'\rceil\le\delta$ times, we obtain
$\kappa_k=\frac1{n'}\sum_{r=1}^{n'-1}\tau_r\bigl(1-e^{-2\pi ir/n'}\bigr)$ with
$\tau_r\in[0,1]$, i.e., $\kappa_k\in S_{n'}$ by Proposition~\ref{p:zono}. Since $\lambda_m$
is a vertex of $S_n$ and $S_{n'}\subseteq S_n$, the point $\lambda_m$ could belong to
$S_{n'}$ only as a vertex of $S_{n'}$. By Proposition~\ref{p:nested}(ii),(iii), no vertex
of $S_{n'}$ other than $0$ and $1$ is a vertex of $S_n$. Hence $\kappa_k\ne\lambda_m$.
\qed\end{pf}

\begin{proposition}[first-order stability]\label{p:first}
Let $1\le m\le n-2$, let $M\in\cL_n$, and let $\lambda(t)$ be the eigenvalue of
$(1-t)\LT_m(0)+tM\in\cL_n$ with $\lambda(0)=\lambda_m$ \textup{(}Lemma~\textup{\ref{l:simple})}.
Then
\[
\lambda_m+t\,\lambda'(0)\in S_n\qquad\text{for all }t\in[0,1].
\]
Moreover, $d_m\bigl(\lambda'(0)+\lambda_m\bigr)\le0$, with equality only if $M$ and
$\LT_m(0)$ differ only in entries $(k,k+m+1)$, and
$d_{m-1}\bigl(\lambda'(0)+\lambda_m\bigr)\le0$, with equality only if they differ only in
entries $(k,k+m)$.
\end{proposition}

\begin{pf}
Write $a_{kj}$ and $b_{kj}$ for the negated off-diagonal entries of $\LT_m(0)$ and $M$, and
put $\rho_s=\sum_{k=1}^n\bigl(b_{k,k+s}-a_{k,k+s}\bigr)$, $1\le s\le n-1$. Since
$a_{k,k+s}=\frac1n$ for $s\le m$ and $a_{k,k+s}=0$ for $s>m$, we have $\rho_s\in[-1,0]$
for $s\le m$ and $\rho_s\in[0,1]$ for $s>m$. The matrix $\LT_m(0)$ is circulant, hence
normal, and $f_1$ is both a right and a left eigenvector for the simple eigenvalue
$\lambda_m$. Therefore
\[
\lambda'(0)=\frac{f_1^{*}\bigl(M-\LT_m(0)\bigr)f_1}{f_1^{*}f_1}
=\frac1n\sum_{k}\sum_{j\ne k}\bigl(b_{kj}-a_{kj}\bigr)\bigl(1-\mu^{\,j-k}\bigr)
=\sum_{s=1}^{n-1}\rho_s\,g_s .
\]
Hence
$\lambda_m+t\lambda'(0)=\sum_{s\le m}(1+t\rho_s)\,g_s+\sum_{s>m}t\rho_s\,g_s$ with all
coefficients in $[0,1]$ for $t\in[0,1]$, which lies in $S_n$ by
Proposition~\ref{p:zono}.

By the proof of Lemma~\ref{l:edgem}, $\operatorname{Re}(\bar\nu_m g_s)$ is positive for
$s\le m$, zero for $s=m+1$, and negative for $s\ge m+2$. Thus every term of
$d_m\bigl(\lambda_m+\lambda'(0)\bigr)=\sum_s\rho_s\operatorname{Re}(\bar\nu_m g_s)$ is
nonpositive, and the sum vanishes only if $\rho_s=0$ for all $s\ne m+1$. Since the
differences $b_{k,k+s}-a_{k,k+s}$ have a fixed sign for each $s$, this means that
$b_{k,k+s}=a_{k,k+s}$ for all $k$ and all $s\ne m+1$. The argument for $d_{m-1}$, with
$\nu_{m-1}$ in place of $\nu_m$ and $d_0(z)=\operatorname{Re}(\bar\nu_0z)$ for $m=1$, is the
same with $s=m$ singled out.
\qed\end{pf}

Thus, to first order, no perturbation of $\LT_m(0)$ within $\cL_n$ moves $\lambda_m$ out
of $S_n$, and the only perturbations that are tangent to the edge $[\lambda_m,\lambda_{m+1}]$
are those of~\eqref{e:Lmw}, which change only the arcs $k\to k+m+1$. The counterexamples of
this section are therefore necessarily second-order effects, and their smallness is no
accident.

\begin{proposition}\label{p:pertm}
Let $1\le m\le n-2$ and $\omega\in\R^n$. The eigenvalue $\lambda(t)$ of $\LT_m(t\omega)$ with $\lambda(0)=\lambda_m$ is analytic near $t=0$ and satisfies
\[
\lambda(t)=\lambda_m+t\,\bar\omega\,g_p+O(t^2),\qquad
d_m\bigl(\lambda(t)\bigr)=t^2\sum_{j=1}^{n-1}c_j\,|\hat\omega(j)|^2+O(t^3),
\]
where
\begin{equation}\label{e:cj}
c_j=\frac2n\sin\frac{\pi p}{n}\;
\operatorname{Im}\frac{1-\mu^{p(j+1)}}{F_m(j+1)-F_m(1)} .
\end{equation}
In particular, $c_{n-1}=0$,
\begin{equation}\label{e:c1}
c_1=\frac2n\,\sin\frac{2\pi}{n}\,\sin\frac{\pi}{n}\;
\frac{\sin\bigl(2\pi(m+1)/n\bigr)}{\sin(\pi m/n)},
\end{equation}
and, for even $n$,
\[
c_{n/2}=\frac{c_1}2\quad\text{if $m$ is even},\qquad
c_{n/2}=-\frac2n\,\sin\frac{2\pi}{n}\,\sin\frac{\pi(m+1)}{n}<0\quad\text{if $m$ is odd}.
\]
If $\omega_k\ge0$ for all $k$, then $\LT_m(t\omega)\in\cL_n$ for all sufficiently small $t\ge0$.
\end{proposition}

\begin{pf}
As in the proof of Lemma~\ref{l:pert}, $F^{-1}\LT_m(t\omega)F=R+tE$ with
$R=\operatorname{diag}(\kappa_k)$ and $e_{lk}=\frac1n(1-\mu^{pk})\hat\omega(k-l)$. By
Lemma~\ref{l:simple}, the Schur complement argument applies and gives
$\lambda(t)=\kappa_1+te_{11}+t^2\beta+O(t^3)$ with $e_{11}=\bar\omega g_p$ and
\[
\beta=\sum_{k\ne1}\frac{e_{1k}e_{k1}}{\kappa_1-\kappa_k}
=\frac{1-\mu^p}{n}\sum_{j=1}^{n-1}\frac{1-\mu^{p(j+1)}}{F_m(j+1)-F_m(1)}\,|\hat\omega(j)|^2,
\]
because $\kappa_1-\kappa_k=\frac1n\bigl(F_m(k)-F_m(1)\bigr)$. Since
$1-\mu^p=2i\sin\frac{\pi p}{n}\,e^{-i\pi p/n}$, we have
$\bar\nu_m(1-\mu^p)=-2i\sin\frac{\pi p}n$, and $\operatorname{Re}(-iz)=\operatorname{Im}z$
gives~\eqref{e:cj}. The first-order term does not contribute to $d_m$, since $g_p$ is
parallel to the edge.

For $j=n-1$, the numerator in~\eqref{e:cj} vanishes. For $j=1$, using
$F_m(k)=\frac{1-\mu^{pk}}{1-\mu^k}$ for $\mu^k\ne1$, we get
$F_m(2)-F_m(1)=\frac{(1-\mu^p)(\mu^p-\mu)}{(1-\mu)(1+\mu)}$, hence
\[
\frac{1-\mu^{2p}}{F_m(2)-F_m(1)}=\frac{(1+\mu^p)(1-\mu^2)}{\mu^p-\mu}
=-\frac{2\cos\frac{\pi p}n\,\sin\frac{2\pi}n}{\sin\frac{\pi m}n}\,e^{-i\pi/n},
\]
whose imaginary part is
$2\cos\frac{\pi p}n\sin\frac{2\pi}n\sin\frac\pi n\big/\sin\frac{\pi m}n$. Multiplying by
$\frac2n\sin\frac{\pi p}n$ gives~\eqref{e:c1}.

For even $n$ and $j=n/2$, we have $\mu^{j+1}=-\mu$ and $F_m(j+1)=\sum_{s=0}^m(-\mu)^s$. If
$m$ is even, the same computation gives exactly half the value obtained for $j=1$. If $m$
is odd, then $1-\mu^{p(j+1)}=1-\mu^p$ and $F_m(j+1)=\frac{1-\mu^p}{1+\mu}$, so
$F_m(j+1)-F_m(1)=-\frac{2\mu(1-\mu^p)}{1-\mu^2}$ and the quotient in~\eqref{e:cj} equals
$-\frac{1-\mu^2}{2\mu}=-i\sin\frac{2\pi}n$, which yields the stated negative value.
\qed\end{pf}

\begin{theorem}\label{t:edges}
Let $m\ge1$ and $2(m+1)<n$. For $\omega_k=1+\cos\frac{2\pi k}{n}$, the eigenvalue
$\lambda(t)$ of $\LT_m(t\omega)$ with $\lambda(0)=\lambda_m$ satisfies
\[
d_m\bigl(\lambda(t)\bigr)=\frac{c_1}4\,t^2+O(t^3),\qquad c_1>0,
\]
so $\lambda(t)$ lies beyond the edge $[\lambda_m,\lambda_{m+1}]$\, for all sufficiently small
$t>0$. If, in addition, $n$ and $m$ are even, the same holds for the alternating weights
$\omega_k=\frac12\bigl(1+(-1)^k\bigr)$, with $d_m(\lambda(t))=\frac{c_1}8\,t^2+O(t^3)$.
\end{theorem}

\begin{pf}
For $\omega_k=1+\cos\frac{2\pi k}{n}$, the only nonzero $\hat\omega(j)$ with $j\ne0$ are
$\hat\omega(1)$ and $\hat\omega(n-1)$, of modulus $\frac12$; since $c_{n-1}=0$,
Proposition~\ref{p:pertm} gives the expansion, and $c_1>0$ by~\eqref{e:c1}, since
$0<\frac{2\pi(m+1)}{n}<\pi$. For the alternating weights, the only nonzero $\hat\omega(j)$
with $j\ne0$ is $\hat\omega(n/2)$, of modulus $\frac12$, and $c_{n/2}=\frac{c_1}2$.
\qed\end{pf}

For $m=1$, \eqref{e:c1} reduces to the coefficient $\frac4n\sin^2\frac{2\pi}n\cos\frac{2\pi}n$
of the first harmonic in Lemma~\ref{l:pert}.

\begin{corollary}\label{c:edges}
Let $n\ge5$. For every edge $E$ of $S_n$ other than $[0,\lambda_1]$, $[\lambda_{n-2},1]$,
their conjugates, and, for even $n$, the horizontal edges $\bigl[\lambda_{n/2-1},\lambda_{n/2}\bigr]$
and $\bigl[\bar\lambda_{n/2-1},\bar\lambda_{n/2}\bigr]$, there are points of $\Lambda_n$ lying beyond
the supporting line of $S_n$ through~$E$.
\end{corollary}

\begin{pf}
For the edges $[\lambda_m,\lambda_{m+1}]$ with $1\le m<\frac n2-1$, this is
Theorem~\ref{t:edges}. The map $z\mapsto1-\bar z$ preserves $S_n$ and~$\Lambda_n$
(Remark~\ref{r:sym}) and sends $[\lambda_m,\lambda_{m+1}]$ to
$[\lambda_{n-2-m},\lambda_{n-1-m}]$, which covers the edges with
$\frac n2-1<m\le n-3$. Conjugation covers the lower edges. For odd $n$, no value of $m$
is left out other than $m=0$ and $m=n-2$.
\qed\end{pf}

\begin{rmk}[the horizontal edge]\label{r:central}
Let $n$ be even and $m=\frac n2-1$, so that $p=\frac n2$ and $\mu^p=-1$. Then the numerator
in~\eqref{e:cj} vanishes for odd~$j$. For even $j\le n-2$, the computation in the proof of
Proposition~\ref{p:pertm} gives
$c_j=-\frac4n\sin\frac\pi n\,\sin\frac{\pi(j+1)}n\big/\sin\frac{\pi j}n<0$. Thus, at the
horizontal edge, the second-order term never points outward, whatever the weights~$\omega$.
For $n=4$, this edge is $[\lambda_1,\lambda_2]$, in accordance with
Remark~\ref{r:gamma}. \nH{Liu~\nvv{\cite[Proposition~4]{Liu026v3}} extended this observation to every $C^2$ curve of standardized Laplacians through $\LT_m(0)$: along such a curve, $\operatorname{Im}\lambda$ does not exceed the height of the horizontal edge up to $o(t^2)$.} Whether the horizontal edge lies on $\partial\Lambda_n$ is the subject of Question~\ref{q:top} below.
\end{rmk}

\begin{rmk}[one-sided versus alternating chords]\label{r:parity}
The sign pattern of the coefficients $c_j$ explains the shape of numerically optimal
chord weights. For $m=1$, $c_j+c_{n-j}$ is proportional to $\cos\frac{2\pi j}{n}$, so low-frequency harmonics favour slowly varying,
``one-sided'' weight patterns, as for $\Gamma_7$ and in Theorem~\ref{t:five}; the alternating harmonic $j=n/2$ pushes the eigenvalue inward. For
even $m\ge2$ and even $n$, the alternating harmonic pushes outward with half the
coefficient of the first harmonic, but it is realized by $0/1$ weights without any other
harmonics, whereas one-sided weights also excite harmonics $j\ge2$ with negative
coefficients. Accordingly, in the cases we computed, the second-order coefficient is maximized among
all $0/1$ weight vectors by alternating chords for $(n,m)=(8,2)$, $(10,2)$, $(12,2)$, $(12,4)$,
and by one-sided blocks of consecutive chords for $(n,m)=(7,1)$, $(7,2)$, $(8,1)$, $(9,2)$, $(10,3)$. Since alternating weights commute with $Q^2$, the
eigenvalue problem then reduces to a $2\times2$ matrix on $\operatorname{span}\{f_1,f_{1+n/2}\}$.
For $(n,m)=(8,2)$, numerical optimization over all $\omega\in[0,1]^8$ gives weights alternating between $0$ and about $0.63$, i.e., chords $k\to k+3$ of weight about $0.63$ at every other vertex, and a distance of about $9.8\cdot10^{-4}$ beyond $[\lambda_2,\lambda_3]$; for $(n,m)=(7,2)$, it gives one-sided weights and about $5.9\cdot10^{-4}$.
\end{rmk}
\end{nbblock}

\section[Computations for small orders]{\nb{Computations for small orders}}

\begin{itemize}
\item $n=4$: all $2^{12}$ matrices in $\cL_4$ with off-diagonal entries in $\bigl\{0,-\frac14\bigr\}$, i.e., the standardized Laplacians of all unweighted digraphs on $4$ vertices;
$n=5$: all $2^{20}$ such matrices in~$\cL_5$. All eigenvalues lie in $S_n$ up to $3\cdot10^{-15}$.
The maximal imaginary parts are $\frac14$ and $0.3077683\ldots=\frac1{10}\cot\frac\pi{10}$, the tops of $S_4$ and $S_5$.
\item $n=6$: all $16^6$ such matrices in $\cL_6$ whose digraphs have out-degrees $\le2$. No violation.
\item \nb{$n=4$: numerical searches over weighted digraphs found no violation. In
particular, $300$ local maximizations of $\max\operatorname{Im}\lambda$ were started near each of the standardized Laplacians of the directed $4$-cycle, of the $4$-cycle with all chords $k\to k+2$, of the half-sum of these two matrices, and of $H_4(1)$, and from random matrices in~$\cL_4$; in all runs, the resulting maximum of $\operatorname{Im}\lambda$ agreed with $\frac14$ up to rounding error. Note that $\LT(H_4(1))$, like $\LT_4$ \textup{(}Remark~\textup{\ref{r:path})}, has the eigenvalue $\frac12+\frac14i$, the midpoint of the horizontal edge $[\lambda_1,\lambda_2]$ of~$S_4$; more generally, $\LT(H_4(t))$ has the eigenvalue $\frac{1+t}4+\frac14i$ on this edge. As a check, local maximizations of the signed distance beyond the line through
$[\lambda_1,\lambda_2]$ over all off-diagonal entries, started from the directed $n$-cycle with random chord weights on $k\to k+2$ and small random weights on the other arcs, found the counterexamples of Section~\textup{\ref{s:chords}} in $28$ of $40$ runs for $n=5$ and in $26$ of $40$ runs for $n=6$, but no violation for $n=4$.}
\end{itemize}

\nb{According to the first item, no unweighted digraph on $5$ vertices has an eigenvalue outside~$S_5$, so the counterexamples for $n=5$ are necessarily weighted. This agrees with Remark~\ref{r:path}: at $t=1$, the branch $\lambda(t)$ of Theorem~\ref{t:allfive} has returned to~$S_5$.}

For $n=4$, $S_4=R_4\cap\bigl\{|\operatorname{Im}z|\le\frac14\bigr\}$, where $R_4$ is the region of~\cite[Theorem~5]{AgaChe05LAA}.
Let, as in~\cite{AgaChe05LAA}, $z(n)=\sup\{|\operatorname{Im}\lambda|:\lambda\in\Lambda_n\}$.
So for $n=4$ the conjecture is equivalent to $z(4)=\frac14$ (against the Pick bound
$\frac18\cot\frac\pi8\approx0.302$ \cite{GregoryKirkland93,Pick1922}). The case $n\le3$ is settled in~\cite{AgaChe05LAA}.

\begin{proposition}[realizations with at most two fractional entries per row]\label{p:sparse}
If $\lambda\in\Lambda_n$, then $\lambda$ is an eigenvalue of some $\LT\in\cL_n$ in which every row has at
most two off-diagonal entries in the open interval $\bigl(-\frac1n,0\bigr)$.
\end{proposition}

\begin{pf}
Let $\LT v=\lambda v$ with $\LT\in\cL_n$ and $v\ne0$, and write $a_{ij}=-\ell_{ij}\in\bigl[0,\frac1n\bigr]$ for $j\ne i$. Since the $i$th row of the
equation $\LT v=\lambda v$ involves only the entries of the $i$th row of $\LT$, the rows can be modified independently.

Fix $i\in\supp v$ and consider the set $P_i$ of all vectors $(a_{ij})_{j\ne i}\in\bigl[0,\frac1n\bigr]^{n-1}$ satisfying $\sum_{j\ne i}a_{ij}(v_i-v_j)=\lambda v_i$, i.e., the two real linear equations given by the real and imaginary parts of this identity. The set $P_i$ is a polytope, and it is nonempty, since it contains the $i$th row of~$\LT$; hence it has a vertex. At a vertex of $P_i$, the active constraints have rank $n-1$; the two equations contribute at most~$2$, so at least $n-3$ coordinates attain one of the bounds $0$ or $\frac1n$. Thus there is a vector in $P_i$ with at most two coordinates in the open interval $\bigl(0,\frac1n\bigr)$. For $i\notin\supp v$, the $i$th row of the equation reads $0=\lambda\cdot0$ and is satisfied by the zero row.

Let $\LT'$ be the matrix whose off-diagonal entries in row $i$ are the negatives
of the chosen vector (the zero vector if $i\notin\supp v$), with diagonal entries determined by zero row sums. Then $\LT'\in\cL_n$, each of its rows has at most two off-diagonal entries in $\bigl(-\frac1n,0\bigr)$, and each row of $\LT'v=\lambda v$ holds
by construction. Hence $\lambda\in\sigma(\LT')$.
\qed\end{pf}

\nH{For $n=4$, a further linear-programming reduction to matrices with at most six fractional off-diagonal entries underlies the computer-assisted proofs that $\Lambda_4=S_4$~\cite{Liu026v3,Marshall026}.}

\section{Further structural properties and questions}

In this section, we establish some structural properties of $\Lambda_n$ and $S_n$ valid for every $n$ and formulate open questions.

\subsection{Relations between different orders}

\begin{proposition}[balanced blow-up]\label{p:blowup}
If $m\mid n$, then $\Lambda_m\subseteq\Lambda_n$.
\end{proposition}

\begin{pf}
Let $\LT=D-A\in\cL_m$ with $\LT x=\lambda x$, and let $d=n/m$. Replace each vertex by $d$ copies, give every arc $(i,s)\to(j,s')$ with $i\ne j$ the entry $-a_{ij}/d\in\bigl[-\frac1n,0\bigr]$, and keep the diagonal $D\otimes I_d$. The resulting matrix $D\otimes I_d-\frac1d A\otimes J_d$ lies in $\cL_n$, and it maps $x\otimes\mathbf 1_d$ to $\lambda\,x\otimes\mathbf 1_d$.
\qed\end{pf}

Consequently the conjecture would have implied $S_m\subseteq S_n$ for $m\mid n$.
This inclusion does hold.

\begin{proposition}\label{p:nested}
Let $m\mid n$, $2\le m<n;$ let $\lambda^{(m)}_k$ and $\lambda^{(n)}_k$ denote the upper-boundary vertices \textup{\eqref{lambda_k}} of $S_m$ and $S_n$. Then\textup{:}
\begin{enumerate}
\item[\textup{(i)}] $S_m\subseteq S_n;$
\item[\textup{(ii)}] if $n=2m$, then each vertex $\lambda^{(m)}_l$, $1\le l\le m-2$, lies in the relative interior of the edge $\bigl[\lambda^{(n)}_{2l},\lambda^{(n)}_{2l+1}\bigr]$ of~$S_n;$
\item[\textup{(iii)}] if $n\ge3m$, then $S_m\setminus\{0,1\}$ lies in the interior of~$S_n$.
\end{enumerate}
\end{proposition}

\begin{pf}
(i) Let $d=n/m$, fix $\theta$, and put $c=\cos\theta$.
By~\eqref{e:hS}, with the vanishing term $j=0$ added,
\[
n\,h_{S_n}(\theta)=\sum_{j=0}^{n-1}\Bigl(c-\cos\bigl(\theta+\tfrac{2\pi j}{n}\bigr)\Bigr)_+ .
\]
Writing $j=r+dl$ with $0\le r\le d-1$ and $0\le l\le m-1$, we obtain
\[
n\,h_{S_n}(\theta)=\sum_{r=0}^{d-1}A\Bigl(\theta+\frac{2\pi r}{n}\Bigr),
\qquad
A(\alpha)=\sum_{l=0}^{m-1}\bigl(c-\cos\alpha_l\bigr)_+,\qquad
\alpha_l=\alpha+\frac{2\pi l}{m},
\]
and $m\,h_{S_m}(\theta)=A(\theta)$. Hence it suffices to show that $A(\alpha)\ge A(\theta)$ for all $\alpha$: then $n\,h_{S_n}(\theta)\ge d\,A(\theta)=n\,h_{S_m}(\theta)$, and since $\theta$ is arbitrary, $S_m\subseteq S_n$.

Since $\sum_l\cos\alpha_l=0$, we have
$\sum_l(c-\cos\alpha_l)=mc$, and
$x_+=\frac12(x+|x|)$ gives
\[
A(\alpha)=\frac12\bigl(mc+D(\alpha)\bigr),
\qquad
D(\alpha)=\sum_{l=0}^{m-1}
|\cos\alpha_l-c|.
\]
Geometrically, $D(\alpha)$ is the sum of the distances from the vertices $e^{i\alpha_l}$ of a regular $m$-gon to the line
$\operatorname{Re}z=c$.

Call $\alpha$ a breakpoint if $\cos\alpha_l=c$ for some $l$.
On every arc between consecutive breakpoints the signs\, $\varepsilon_l=\operatorname{sign}(\cos\alpha_l-c)$\, are constant, and hence
\[
D''(\alpha)
=
-\sum_l\varepsilon_l\cos\alpha_l
=
-2\sum_{\varepsilon_l=1}\cos\alpha_l.
\]
The last sum is nonnegative. If $c\ge0$, each of its terms is
positive. If $c<0$, then
\[
\sum_{\varepsilon_l=1}\cos\alpha_l
=
-\sum_{\varepsilon_l=-1}\cos\alpha_l\ge0,
\]
because $\cos\alpha_l<c<0$ whenever $\varepsilon_l=-1$.
Thus $D$ is concave on every arc between consecutive breakpoints. Since it extends continuously to the endpoints, its minimum on each such closed arc is attained at an endpoint.
Hence the global minimum of $D$ is attained at a breakpoint.

At a breakpoint,
\[
\alpha\equiv\theta
\quad\text{or}\quad
\alpha\equiv-\theta
\pmod{2\pi/m}.
\]
The corresponding configurations are mirror images under complex conjugation, which preserves $D$. Therefore,\,
$\min_\alpha D(\alpha)=D(\theta),$\, and hence $A(\alpha)\ge A(\theta)$.

(ii) Let $n=2m$ and $1\le l\le m-2$. Since
$g_{2l+1}=\lambda^{(n)}_{2l+1}-\lambda^{(n)}_{2l}$ is the edge vector of $\bigl[\lambda^{(n)}_{2l},\lambda^{(n)}_{2l+1}\bigr]$, it suffices to show
\begin{equation}\label{e:inscribed}
\lambda^{(m)}_l=\lambda^{(n)}_{2l}+t_l\,g_{2l+1}
\quad\text{with \,} 0<t_l<1,
\quad\text{where}\quad
t_l=\frac{\sin\frac{2\pi l}{n}}{\sin\frac{2\pi l}{n}+\sin\frac{(2l+2)\pi}{n}}.
\end{equation}

The inequalities $0<t_l<1$ hold because $0<\frac{2\pi l}{n}<\frac{(2l+2)\pi}{n}<\pi$, so both sines in~\eqref{e:inscribed} are positive.

To prove the equality in~\eqref{e:inscribed}, let $\mu=e^{-2\pi i/n}$, so that $\mu^2=e^{-2\pi i/m}$. By the definition of the vertices,
\begin{equation}\label{e:diff}
\lambda^{(m)}_l-\lambda^{(n)}_{2l}
=\frac1n\Bigl(2\sum_{j=1}^{l}(1-\mu^{2j})-\sum_{k=1}^{2l}(1-\mu^k)\Bigr)
=\frac{1-\mu}n\sum_{j=1}^{l}\mu^{2j-1}
=\frac{\mu(1-\mu^{2l})}{n(1+\mu)} .
\end{equation}
Using $1-e^{-i\varphi}=2\sin\frac\varphi2\,e^{i(\pi/2-\varphi/2)}$ and $1+\mu=2\cos\frac\pi n\,e^{-i\pi/n}$, we write both the right-hand side of~\eqref{e:diff} and $g_{2l+1}=\frac1n(1-\mu^{2l+1})$ in polar form:
\[
\frac{\mu(1-\mu^{2l})}{n(1+\mu)}
=\frac{\sin\frac{2\pi l}{n}}{n\cos\frac\pi n}\,e^{\,i\bigl(\frac\pi2-\frac{(2l+1)\pi}{n}\bigr)},
\qquad
g_{2l+1}=\frac2n\sin\frac{(2l+1)\pi}{n}\,e^{\,i\bigl(\frac\pi2-\frac{(2l+1)\pi}{n}\bigr)}.
\]
The arguments coincide, and the ratio of the moduli is $\sin\frac{2\pi l}{n}\big/\bigl(2\cos\frac\pi n\sin\frac{(2l+1)\pi}{n}\bigr)$, which
equals $t_l$ because $\sin\frac{2\pi l}{n}+\sin\frac{(2l+2)\pi}{n}=2\cos\frac\pi n\sin\frac{(2l+1)\pi}{n}$. This proves~\eqref{e:inscribed}.
By conjugation, the analogous statement holds for the lower-boundary vertices.

(iii) If $m=2$, then $S_2=[0,1]$, and the assertion follows from $f_n(x)>0$ for $0<x<1$. Assume $m\ge3$. Let $n\ge3m$, and suppose, to the contrary, that some $p\in S_m\setminus\{0,1\}$ lies on $\partial S_n$. A supporting line of $S_n$
through $p$ also supports $S_m$ at $p$, since $S_m\subseteq S_n$. Hence, for its outward normal $e^{i\theta}$,
\begin{equation}\label{e:contact}
h_{S_m}(\theta)=h_{S_n}(\theta)=\operatorname{Re}(e^{-i\theta}p).
\end{equation}
Let $E$ be the set of maximizers of $\operatorname{Re}(e^{-i\theta}z)$ on $S_m$; then $p\in E$. We show that \eqref{e:contact} is impossible, considering two cases.

\emph{Case}~1: $0\in E$ or $1\in E$. Since also $p\in E\setminus\{0,1\}$, the set $E$ is an edge of $S_m$ incident to $0$ or~$1$. By the symmetries
$z\mapsto\bar z$ and $z\mapsto1-\bar z$ of $S_m$ and $S_n$, we may assume that $E=\bigl[0,\lambda^{(m)}_1\bigr]$. The edges of $S_m$ and $S_n$ issuing from $0$ into the upper half-plane have arguments $\frac\pi2-\frac\pi m$ and $\frac\pi2-\frac\pi n$, respectively. Since $\frac\pi2-\frac\pi n>\frac\pi2-\frac\pi m$, the edge $[0,\lambda^{(n)}_1]$ of $S_n$ leaves $0$ on the outer side of the line through $E$; hence this line does not support $S_n$, contrary to \eqref{e:contact}.

\emph{Case}~2: $0,1\notin E$. We use the notation of the proof of~(i): $c=\cos\theta$, $m\,h_{S_m}(\theta)=A(\theta)=\frac12\bigl(mc+D(\theta)\bigr)$,
$n\,h_{S_n}(\theta)=\sum_{r=0}^{d-1}A(\theta+2\pi r/n)$, and
$\min_\alpha D(\alpha)=D(\theta)$. Since $0\notin E$ and $1\notin E$, we have
$h_{S_m}(\theta)\ne0$ and $h_{S_m}(\theta)\ne\cos\theta$, i.e., $D(\theta)\ne\pm mc$.

We claim that $D(\alpha)>D(\theta)$ whenever $\alpha\not\equiv\pm\theta\pmod{2\pi/m}$.
Indeed, such an $\alpha$ is not a breakpoint, so it lies inside an arc on which $D$ is concave. If $D(\alpha)=D(\theta)=\min D$, then $D$ is constant on this arc.
But $D$ is constant on an arc only if all vertices $e^{i\alpha_l}$ lie on one side of the line $\operatorname{Re}z=c$, i.e., only if $D=\pm mc$ there. This would give $D(\theta)=\pm mc$, a contradiction.

Now, for $r=1,\dots,d-1$, we have $\theta+2\pi r/n\not\equiv\theta$, and $\theta+2\pi r/n\equiv-\theta \pmod{2\pi/m}$ for at most one such~$r$. Since
$d\ge3$, there is an $r$ with $A(\theta+2\pi r/n)>A(\theta)$, and therefore $n\,h_{S_n}(\theta)>d\,A(\theta)=n\,h_{S_m}(\theta)$, contrary to \eqref{e:contact}.

Thus no point of $S_m\setminus\{0,1\}$ lies on $\partial S_n$, which proves~(iii).
\qed\end{pf}

Without the divisibility assumption the inclusion can fail; for example, $S_3\not\subseteq S_4$.

\begin{rmk}
Proposition~\ref{p:blowup} yields a natural nested sequence $\Lambda_{k!}\subseteq\Lambda_{(k+1)!}$ and the equality
\[
\bigcup_{n\ge2}\Lambda_n=\bigcup_{k\ge2}\Lambda_{k!}.
\]
Thus the large-order geometry of $\bigcup_{n\ge2}\Lambda_n$ may be studied through the increasing family $(\Lambda_{k!})_{k\ge2}$.
\end{rmk}

\begin{corollary}\label{c:union}
$\bigcup_{n\ge2}S_n=\bigcup_{k\ge2}S_{k!};$ \,the set $U=\bigl(\bigcup_{n\ge2}S_n\bigr)\setminus\{0,1\}$ is open and convex.
\end{corollary}

\begin{pf}
If $p\in S_m\setminus\{0,1\}$ with $m\ge2$, then $p$ lies in the interior of $S_{3m}$ by Proposition~\ref{p:nested}(iii), and this interior is contained in~$U$; hence $U$ is open. By Proposition~\ref{p:nested}(i), $S_n\subseteq S_{N}$ whenever $n\mid N$, so $\bigcup_{n\ge2}S_n=\bigcup_{k\ge2}S_{k!}$ is the union of an increasing chain of convex sets and is therefore convex. Finally, $0$ and $1$ are extreme points of this union, since it lies in the strip $0\le\operatorname{Re}z\le1$ and meets the lines $\operatorname{Re}z=0$ and $\operatorname{Re}z=1$ only at $0$ and~$1$; removing them preserves convexity.
\qed\end{pf}

\subsection{Geometry of $\Lambda_n$}

Recall that a subset of $\R^d$ is called \emph{semialgebraic\/} if it is a finite union of sets, each defined by finitely many polynomial equalities and inequalities.

\begin{proposition}\label{p:geometry}
For all $n\ge2$, the set $\Lambda_n$ is compact and semialgebraic. It is star-shaped with respect to both $0$ and $1$.
\end{proposition}

\begin{pf}
If $\lambda\in\Lambda_n$ is realized by $\LT\in\cL_n$, then $t\lambda$ is realized by $t\LT\in\cL_n$ for every $0\le t\le1$.
Thus $\Lambda_n$ is star-shaped with respect to $0$. By the symmetry $\Lambda_n=1-\Lambda_n$ from Remark~\ref{r:sym}, it is also star-shaped with respect to $1$.

For semialgebraicity, write $\lambda=x+iy$ and an eigenvector as $v=u+iw$, normalized by
\[
\sum_i\bigl(u_i^2+w_i^2\bigr)=1.
\]
Let $\LT(a)$ denote the Laplacian determined by the variables $a_{ij}$. The conditions
\[
0\le a_{ij}\le\frac1n,\qquad
\LT(a)u=xu-yw,\qquad
\LT(a)w=yu+xw
\]
form a finite system of polynomial equalities and inequalities in real variables. Hence $\Lambda_n$, obtained by projection onto
$(x,y)$, is semialgebraic by the Tarski--Seidenberg theorem \cite{BochnakCosteR98AG}.
The set $\cL_n$ is compact, and all its eigenvalues lie in a common compact disk. Hence
\[
\{(\LT,\lambda):\LT\in\cL_n,\ \det(\lambda I-\LT)=0\}
\]
is compact, and its projection onto the $\lambda$-coordinate is $\Lambda_n$. Thus $\Lambda_n$ is compact.
\qed\end{pf}

The boundary $\partial\Lambda_n$ is semialgebraic and therefore admits a finite stratification into smooth semialgebraic arcs and points; see~\cite{BochnakCosteR98AG}.

For a set $K\subseteq\C$, let
\[
\ker_\star K
=
\{a\in K:[a,z]\subseteq K\text{ for every }z\in K\}
\]
denote its \emph{star kernel}.

\begin{corollary}\label{c:profile}
For every $n\ge2$\textup{:}
\begin{enumerate}
\item[\textup{(i)}]   $[0,1]\subseteq\ker_\star\Lambda_n.$
\item[\textup{(ii)}]  There is a function $\eta_n:[0,1]\to[0,\infty)$ such that\, $\Lambda_n=\{x+iy:0\le x\le1,\ |y|\le \eta_n(x)\}.$
\item[\textup{(iii)}] $\eta_n$ is continuous$,$ $\eta_n(x)=\eta_n(1-x),$\, and, for\, $0<x<y<1,$
\[
\frac{1-y}{1-x}\,\eta_n(x)
\le \eta_n(y)
\le
\frac{y}{x}\,\eta_n(x).
\]
\item[\textup{(iv)}] $\eta_n(x)/x$ is nonincreasing and $\eta_n(x)/(1-x)$ is nondecreasing.
\end{enumerate}
\end{corollary}

\begin{pf}
(i) If $a,b\in\ker_\star K$ and $z\in K$, then every point of $\conv\{a,b,z\}$ lies on a segment $[a,w]$ with $w\in[b,z]$ and $w\in K$; hence $\conv\{a,b,z\}\subseteq K$. Thus $[a,b]\subseteq\ker_\star K$, i.e., the star kernel is convex. By Proposition~\ref{p:geometry}, $\{0,1\}\subseteq\ker_\star\Lambda_n$; therefore $[0,1]\subseteq\ker_\star\Lambda_n$.

(ii) The Gershgorin disks give $\operatorname{Re}\lambda\ge0$ for
$\lambda\in\Lambda_n$, and the symmetry
$\Lambda_n=1-\Lambda_n$ gives $\operatorname{Re}\lambda\le1$.
Together with conjugation symmetry, compactness, and $[0,1]\subseteq\ker_\star\Lambda_n$, this gives the stated representation of $\Lambda_n$ by $\eta_n$.

(iii) For $z=x+i\eta_n(x)$, the triangle $\conv\{0,1,z\}$ lies in $\Lambda_n$. Reading its height at $y>x$ gives the first inequality. Applying the same argument to $y+i\eta_n(y)$ and reading the height at $x$ gives the second. The symmetry $\eta_n(x)=\eta_n(1-x)$ follows from $\Lambda_n=1-\Lambda_n=\overline{\Lambda_n}$.
The displayed inequalities give continuity on $(0,1)$.
All Gershgorin disks of matrices in $\cL_n$ lie in the disk $|z-\rho|\le\rho$, $\rho=1-\frac1n$; hence $\eta_n(x)\le\sqrt{2\rho x-x^2}\to0=\eta_n(0)$ as $x\to0^+$, and, by symmetry, $\eta_n$ is also continuous at~$1$.

(iv) This follows from the inequalities in (iii).
\qed\end{pf}

By Section~\ref{s:polygzono}, $S_n=\{x+iy: 0\le x\le1,\ |y|\le f_n(x)\}$, so $f_n\le\eta_n$, and the
conjecture of~\cite{AgaChe05LAA} amounts to $\eta_n=f_n$; Theorem~\ref{t:counter}
shows that $\eta_7>f_7$ near $\operatorname{Re}\lambda^\ast$\nb{, and Theorem~\ref{t:allfive} shows that $\eta_n\ne f_n$ for every $n\ge5$}.
In this representation, $\Lambda_n$ is convex if and only if $\eta_n$ is concave. \nH{By~\nvv{\cite[Corollary~3]{Liu026v3}}, $\eta_n$ is not concave for $n\ge5$.}

\begin{rmk}\label{r:edgesbd}
\nb{By Corollary~\ref{c:profile}(i), $\conv\{0,1,z\}\subseteq\Lambda_n$ for every
$z\in\Lambda_n$. For the points $z=\lambda(t)$ of Theorem~\ref{t:edges} we have
$z-\lambda_m=t\,g_{m+1}+O(t^2)$ and $0<d_m(z)=O(t^2)$; hence the line through
$[\lambda_m,\lambda_{m+1}]$ meets the interior of this triangle in points of the relative
interior of that edge. Together with Corollary~\ref{c:edges}, this shows that every edge
of $S_n$ not listed there contains interior points of~$\Lambda_n$.}
\end{rmk}

We conclude with several open questions.

\subsection{Open questions}

%
\begin{question}
\label{q:shape}Describe $\Lambda_n$ for \nb{$n\ge5$}. \nH{It is not convex~\nvv{\textup{\cite[Corollary~3]{Liu026v3}}}.} 
\nb{Which digraphs realize its boundary near $[\lambda_1,\lambda_2]$, in the spirit of Kirkland--\v{S}migoc's \textup{\cite{KirklandSmigoc022}} description of the matrices realizing $\partial\Theta_n?$ Do Hamiltonian cycles with chords, as in Section~\textup{\ref{s:chords}}, and twin structures suffice \nvv{\tu(see also~\tu{\cite{Che026b})}\tu?}}
\end{question}

\begin{question}
\label{q:top}
\nb{The edges $[0,\lambda_1]$ and $[\lambda_{n-2},1]$ and their conjugates lie on
$\partial\Lambda_n$ by~\textup{\cite[Theorems~5 and~6]{AgaChe05LAA} (}based
on~\textup{\cite{DmitrievDynkin45E,DmitrievDynkin46E})}. By Remark~\textup{\ref{r:edgesbd}}, for $n\ge5$
every other edge of $S_n$, except possibly, for even~$n$, the horizontal edges
$[\lambda_{n/2-1},\lambda_{n/2}]$ and $[\bar\lambda_{n/2-1},\bar\lambda_{n/2}]$, contains
interior points of~$\Lambda_n$. \nH{At the horizontal edges, the second-order term never points outward \textup{(}Remark~\textup{\ref{r:central}} and~\textup{\nvv{\cite[Proposition~4]{Liu026v3}})}.} For even~$n$, do the horizontal edges lie on
$\partial\Lambda_n?$ Moreover, is
$\max\{\operatorname{Im}\lambda:\lambda\in\Lambda_n\}$ equal to\, $\frac1n\cot\frac\pi n?$ \nvv{For $n=4$, both answers are affirmative, since $\Lambda_4=S_4$ \tu{\cite{Liu026v3,Marshall026}}.}}
\end{question}

\begin{question}
Does $\Lambda_n$ converge, in Hausdorff distance, to the same cycloidal region of~\textup{\cite[{\it Theorem}~7]{AgaChe05LAA}} as $S_n?$
\end{question}

\appendix
\renewcommand{\thesection}{\appendixname~\Alph{section}}
\section{Certification and numerical reproduction}\label{a:certification}
\renewcommand{\thesection}{\Alph{section}}
\subsection{Certification of the $n=7$ counterexample}

Recall
\[
\begin{aligned}
&p(y)=y^5-7y^4+20y^3-30y^2+25y-11,\\
&c=0.5696689912737478+0.9377623754517331\,i.
\end{aligned}
\]

Set $r=10^{-16}$. Since the real and imaginary parts of $c$ are rational numbers, the following estimates may be checked by exact rational arithmetic:
\[
|p(c)|<8.17\cdot10^{-16},
\qquad
|p'(c)|>9.75,
\]
and
\[
|p''(c)|<37.70,\qquad
|p^{(3)}(c)|<93.88,\qquad
|p^{(4)}(c)|<150.31,\qquad
|p^{(5)}(c)|=120.
\]
Hence, for $|z|=r$,
\[
\left|
\sum_{k=2}^5\frac{p^{(k)}(c)}{k!}z^k
\right|
<1.89\cdot10^{-31}.
\]
It follows that
\[
\left|
p(c)+
\sum_{k=2}^5\frac{p^{(k)}(c)}{k!}z^k
\right|
<
8.18\cdot10^{-16}
<
9.75\cdot10^{-16}
<
|p'(c)z|.
\]

By Rouch\'e's theorem, $p(c+z)$ and $p'(c)z$ have the same number of zeros in $|z|<r$, namely one. Thus $p$ has a zero $y^\ast$ such
that
\[
|y^\ast-c|<10^{-16}.
\]

It remains to show that $\lambda^\ast=y^\ast/7$ lies outside $S_7$.
Let
$
\mu=e^{-2\pi i/7}.
$
The line containing the edge $[\lambda_1,\lambda_2]$ is a supporting line of $S_7$, and its signed Euclidean distance, positive on the
exterior side, can be written in the $y=7\lambda$ coordinate as
\[
d_7(y)
=
\frac17
\left(
\operatorname{Re}\bigl(\bar\mu(1-y)\bigr)-1
\right).
\]
Indeed,
\[
d_7(y)
=
\operatorname{Re}
\bigl(-\bar\mu(\lambda-\lambda_1)\bigr),
\qquad \lambda=\frac y7,
\]
and $-\mu$ is a unit normal to the edge.

Using, for example, the rational interval bounds
\[
0.6234898018587335
<
\cos\frac{2\pi}{7}
<
0.6234898018587336,
\]
\[
0.7818314824680298
<
\sin\frac{2\pi}{7}
<
0.7818314824680299,
\]
one obtains
\[
2.1130622379564\cdot10^{-4}
<
d_7(c)
<
2.1130622379568\cdot10^{-4}.
\]
Since $d_7$ is $1/7$-Lipschitz as a function of $y$,
\[
\bigl|d_7(y^\ast)-d_7(c)\bigr|\le\frac{r}{7}<1.5\cdot10^{-17},
\qquad\text{so}\qquad
d_7(y^\ast)=2.11306223795\ldots\cdot10^{-4}>0.
\]
Thus $\lambda^\ast\notin S_7$.

Finally, a direct interval evaluation of the orthogonal projection onto the line through $[\lambda_1,\lambda_2]$ gives a parameter in $(0.158,0.160)$, so the foot of the perpendicular lies in the interior of that edge. Hence $d_7(y^\ast)$ is the distance from $\lambda^\ast$ to $S_7$, which proves the stated value.

\begin{nbblock}
\renewcommand{\thesection}{\appendixname~\Alph{section}}
\subsection{Certification of the $n=5$ counterexample}\label{a:cert5}
\renewcommand{\thesection}{\Alph{section}}

Let
\[
q(x)=x^4-12x^3+57x^2-134x+140,\qquad
c=2.00696842206+2.12135484024\,i,\qquad r=10^{-10}.
\]
Since the real and imaginary parts of $c$ are rational, the following bounds can be checked
in exact rational arithmetic, by comparing squared moduli with squares of rational numbers:
\[
|q(c)|<1.7\cdot10^{-10},\qquad |q'(c)|>35.75,\qquad
|q''(c)|<62.17,\qquad |q'''(c)|<56.22,\qquad q^{(4)}=24 .
\]
Hence, for $|z|=r$,
\[
\left|\,q(c)+\sum_{k=2}^{4}\frac{q^{(k)}(c)}{k!}z^k\right|
<1.7\cdot10^{-10}+3.2\cdot10^{-19}
<3.575\cdot10^{-9}
<|q'(c)z|.
\]
By Rouch\'e's theorem, $q$ has exactly one zero $x^\sharp$ with $|x^\sharp-c|<10^{-10}$.
Put $\lambda^\sharp=x^\sharp/10$.

Let $\mu=e^{-2\pi i/5}$. Since $\operatorname{Re}(\bar\mu\lambda_1)=\frac15(\cos\frac{2\pi}5-1)$,
in the coordinate $x=10\lambda$ we have
\[
d\Bigl(\frac x{10}\Bigr)=\frac1{10}\Bigl(\operatorname{Re}\bigl(\bar\mu(2-x)\bigr)-2\Bigr)
=\frac1{10}\Bigl((2-\operatorname{Re}x)\cos\tfrac{2\pi}5+\operatorname{Im}x\,\sin\tfrac{2\pi}5-2\Bigr).
\]
Since $\cos\frac{2\pi}5=\frac{\sqrt5-1}4$ and $\sin\frac{2\pi}5=\frac14\sqrt{10+2\sqrt5}$,
the rational bounds
\[
0.30901699437494<\cos\tfrac{2\pi}5<0.30901699437495,\qquad
0.95105651629515<\sin\tfrac{2\pi}5<0.95105651629516
\]
follow from $(4\cos\frac{2\pi}5+1)^2=5$ and $\cos^2+\sin^2=1$. They give
\[
1.537498334399\cdot10^{-3}<d\Bigl(\frac c{10}\Bigr)<1.537498334401\cdot10^{-3}.
\]
Since $x\mapsto d(x/10)$ is $\frac1{10}$-Lipschitz,
$|d(\lambda^\sharp)-d(c/10)|<10^{-11}$, so $d(\lambda^\sharp)=1.5374983\ldots\cdot10^{-3}>0$,
and $\lambda^\sharp\notin S_5$ by Lemma~\ref{l:edge}.

Finally, an interval evaluation of the orthogonal projection of $\lambda^\sharp$ onto the line
through $\lambda_1$ and $\lambda_2$ gives the parameter
$\operatorname{Re}\bigl(\bar g_2(\lambda^\sharp-\lambda_1)\bigr)/|g_2|^2\in(0.174,0.175)$.
So the foot of the perpendicular lies in the interior of the edge $[\lambda_1,\lambda_2]$, and
$d(\lambda^\sharp)$ is the distance from $\lambda^\sharp$ to $S_5$.
\end{nbblock}

\subsection{Numerical reproduction}

The following Python script reproduces the numerical eigenvalue and
signed distance reported in Theorem~\ref{t:counter}.

\begin{verbatim}
import numpy as np, mpmath as mp
n = 7
A = np.zeros((n, n))                 # vertices 0..5 = cycle, 6 = twin of 0
for i in range(5): A[i, i+1] = 1
A[5, 0] = A[5, 6] = 1; A[6, 1] = 1
L = (np.diag(A.sum(1)) - A) / n
print(np.linalg.eigvals(L))

mp.mp.dps = 50
roots, err = mp.polyroots([1, -7, 20, -30, 25, -11], error=True,
                          maxsteps=200, extraprec=200)
y = [r for r in roots if 0 < mp.re(r) < 1 and mp.im(r) > 0][0]
lam = y / 7
mu = mp.exp(-2j*mp.pi/7); l1 = (1-mu)/7; l2 = l1 + (1-mu**2)/7
e = l2 - l1; nr = mp.mpc(-e.imag, e.real)
print(lam, err, mp.re((lam-l1)*mp.conj(nr))/abs(nr))   # > 0: outside S_7
\end{verbatim}

\begin{nbblock}
The next script does the same for Theorem~\ref{t:five}.
\begin{verbatim}
import numpy as np
n, t = 5, 0.5
A = np.zeros((n, n))
for k in range(n): A[k, (k+1) % n] = 1          # Hamiltonian cycle
A[0, 2] = A[1, 3] = t                           # chords 1->3, 2->4
L = (np.diag(A.sum(1)) - A) / n
mu = np.exp(-2j*np.pi/n); l1 = (1 - mu)/n
ev = np.linalg.eigvals(L)
print(ev, ((ev - l1)*np.conj(-mu)).real.max())  # > 0: outside S_5
\end{verbatim}
\end{nbblock}

\Up{.9}
\section*{Acknowledgments}
The author thanks Rafig Agaev for his long-term cooperation in the study of digraph Laplacians and Manjunath Hariprasad for numerous discussions regarding the Laplacian spectra that promise new insights.

\Up{.4}
\section*{Declaration of competing interest}
There is no competing interest.

\Up{.4}
\section*{Declaration of generative AI and AI-assisted technologies in the manuscript preparation process}
In preparing this work, the author used Claude (Anthropic) and ChatGPT (OpenAI) in order to perform eigenvector calculations, verify derivations, and present the results as precisely as possible. The author has thoroughly checked all AI-generated material and bears full responsibility for the content of the article.

\Up{.4}
\bibliographystyle{cas-model2-names}
\bibliography{all2}

\end{document}